\documentclass[10pt,draft,reqno]{amsart}
     \makeatletter
     \def\section{\@startsection{section}{1}%
     \z@{.7\linespacing\@plus\linespacing}{.5\linespacing}%
     {\bfseries
     \centering
     }}
     \def\@secnumfont{\bfseries}
     \makeatother
\newtheorem{theorem}{Theorem}[section]
\newtheorem{lemma}[theorem]{Lemma}
\newtheorem{prop}[theorem]{Proposition}
\newtheorem{corollary}[theorem]{Corollary}
\theoremstyle{definition}

\theoremstyle{remark}

\numberwithin{equation}{section}
\usepackage{url}
\usepackage{amsthm} 
\usepackage{amsmath} 
\usepackage{amsfonts}  
\usepackage{amssymb}
\usepackage{ifpdf}
\ifpdf
  \usepackage[pdftex]{graphicx}
    \usepackage[pdftex,
bookmarks=true,
bookmarksnumbered=true,
hypertexnames=false,
breaklinks=true,
linkbordercolor={0 0 1},
pdfborder={0 0 1}]{hyperref}
 \else
   \usepackage{graphicx}
 \fi

\newcommand{\ch}{\mbox {\bf 1}}

\newcommand{\R}{{\mathbb{R}}} %
\newcommand{\C}{{\mathbb{C}}}  
\newcommand{\Z}{{\mathbb{Z}}} 
\newcommand{\N}{{\mathbb{N}}}  
\newcommand{\I}{{\mathfrak I}}
  
\newcommand{\ol}[1]{\overline{#1}}  
\newcommand{\im}{\mbox{\rm Im }}

\newcommand{\sgn}{\mbox{sgn}}
\newcommand{\cal}{\mathcal} 
 
\newcommand{\scp}[1]{\langle#1\rangle}
\newtheorem{CMth}{Theorem}

\newcounter{cgm}

\begin{document}
\title[Subcritical Homopolymer]{The Subcritical Phase for a  Homopolymer Model}

\author[Iddo Ben-Ari]{Iddo Ben-Ari*}
\thanks{*This work was partially supported by a grant from the Simons Foundation (\#208728 to Iddo Ben-Ari) }
\address{Iddo Ben-Ari: Department of Mathematics, University of Connecticut, Storrs, CT 06269-1009, USA}
\email{iddo.ben-ari@uconn.edu}
\urladdr{iddo.ben-ari.uconn.edu}

\author{Hugo Panzo}
\address{Hugo Panzo: Department of Mathematics, University of Connecticut, Storrs, CT 06269-1009, USA}
\email{hugo.panzo@uconn.edu}
\urladdr{www.math.uconn.edu/~panzo}


\begin{abstract}
We study a model of continuous-time nearest-neighbor random walk on $\Z^d$ penalized by its occupation time at the origin, also known as a homopolymer. For a fixed  real parameter $\beta$ and time $t>0$, we consider the probability measure on paths of the random walk starting from the origin whose Radon-Nikodym derivative is proportional to the exponent of the product $\beta$ times the  occupation time at the origin up to time $t$. The case $\beta>0$ was studied previously by Cranston and Molchanov \cite{CM}\cite{CM2}.  We consider the case $\beta<0$, which is intrinsically different only when the underlying walk is recurrent, that is $d=1,2$.  Our main result is a scaling limit for the distribution of the homopolymer on the time interval $[0,t]$, as $t\to\infty$, a result that coincides   with the scaling limit for  penalized Brownian motion due to Roynette and Yor \cite{roy_yor}.  In two dimensions, the penalizing effect is asymptotically diminished,  and the homopolymer scales  to standard Brownian motion. Our approach is based on  potential analytic and  martingale approximation for the model. We also apply our main result to recover a scaling limit for a wetting model. We study the model through analysis of resolvents. 
\end{abstract} 

\maketitle
\section{Introduction}
\subsection{Description of the model} 
 Let $\Omega$ denote the space of cadlag processes on $\Z^d$. Elements in $\Omega$ are  functions $\omega: \R_+ \to \Z^d$ which are right continuous with   left limits.  We denote the canonical process associated to $\Omega$ by $X=\{X(t):t\in \R_+\}$, where $X(t):=X(t)(\omega)=\omega(t)$, and for $t\in \R_+$ we define ${\cal F}_t$ as  as the $\sigma$ algebra on $\Omega$  generated by the (pre-images of the) coordinate mappings $\omega \to X(s)$, $s\le t$, that is, ${\cal F}_t$ is the smallest $\sigma$-algebra making all coordinate mappings $X(s),~s\le t$, measurable. We also let ${\cal F}$ denote the $\sigma$-algebra generated by $\cup_{t\in \R_+} {\cal F}_t$. For $x\in \Z^d$, let $P_x$ denote the probability distribution on $\cal F$ corresponding to continuous-time nearest-neighbor symmetric random walk on $\Z^d$, with constant jump rate $1$ from each site, conditioned on $X(0)=x$. The  corresponding expectation operator will be denoted by  $E_x$.
 
  For each  {\it parameter} $\beta \in \R$, we define the corresponding {\it homopolymer} as a family of Gibbs measures on $\cal F$, 
 $\{P_{\beta,t}:t\in \R_+\}$ by letting 
   $$ \frac{d P_{\beta,t}}{dP_0}  =\frac{1}{Z_{\beta,t}}  \exp (\beta \I (t) ), $$
   where $\I(t)=\int_0^t \delta_0 (X(s)) ds$ and $Z_{\beta,t}=E_0 [\exp (\beta \I(t) )] $ is the normalizing constant, known as the {\it partition function}, viewed as function of the time parameter $t$. Below, we refer to $P_0$ as the {\it reference measure}. Note that unless $\beta=0$, the homopolymer forms an {\it inconsistent} family of probability measures. One way to view the homopolymer is as a model of random growth of large chains of atoms: starting with a single atom at the origin at time $0$,  atoms are added one at a time, the new atom placed in a site adjacent to the last one added. The random path $\{X(s):s \le t\}$  encodes the configuration of the chain at time $t$ as follows.  Let $J_1=0$ and let  $J_{j+1}=\inf \{t>J_j:X(t^-)\ne X(t)\}$ (with the convention $\inf \emptyset=\infty$). Then $J_j$ and $X(J_j)$ are, respectively,  the time and the site where the $j$-th atom was added, provided $J_j <\infty$. The probabilistic mechanism driving the growth of the homopolymer attempts to capture a simple form of self-interaction, rewarding or penalizing stays at $0$ according to whether $\beta>0$ or $\beta<0$.  A more standard and physically relevant interpretation of the model when $d=1$ \cite[Section 1.2]{giacomin} is as a defect line model, an  interface between media in $1+1$ dimensions, where the graph of the path separates between a medium below it a  medium above it. Note, however, that in \cite{giacomin} the homopolymer considered is in discrete time and the interface is obtained by linear interpolation of the path (in addition, the paths considered are pinned to $0$ at time $t$). A variant of the defect line interpretation  is the wetting model \cite{IsYosh}\cite[Section 1.3]{giacomin} in $1+1$ dimensions. This is obtained from the defect line model by restricting   the Gibbs measure  to paths which do not hit the negative half line.  We analyze the continuous-time version of the wetting model  using our results on the homopolymer. 
  \subsection{First observations}
   A key feature in the area  of polymers and more generally in statistical physics is 
 the phenomenon of a {\it phase transition} in behavior as a function of parameters. The homopolymer exhibits a phase transition, and because it is analytically tractable, allows for rather detailed description of the different phases.  As is customary in models defined through Gibbs measures, the simplest  characterization for phase transition is obtained by gross asymptotic behavior of the partition function as $t\to\infty$, usually according to whether the corresponding {\it Lyapunov exponent}, defined below for the homopolymer, is either zero or is strictly positive. This is because the asymptotic behavior of the partition function typically encapsulates  a fundamental statement on the asymptotic behavior of the Gibbs measures themselves.  We now briefly  illustrate this principle for the homopolymer, and begin by defining the Lyapunov exponent, $\lambda(\beta)$ as 
   $$ \lambda(\beta) = \lim_{t\to\infty} \frac 1t \ln Z_{\beta,t}.$$
Observe that     $\lambda(0)=0$ and for all other values of $\beta$ the limit exists due to a standard sub-additivity argument. Since $X$ is not positive recurrent under the reference measure $P_0$,  it follows that $\lambda (\beta) \ge 0$. To obtain an upper bound on $\lambda(\beta)$, recall the following well-known large deviations estimate: 
     \begin{equation} 
          \label{eq:LDI} 
          \lim_{t\to\infty} \frac{1}{t} \ln P_0(\I (t) \ge \epsilon t) <0\mbox{ for all } \epsilon>0. 
          \end{equation}
          Consequently,  
           $$Z_{\beta,t} \le e^{\beta t } P(\I(t)\ge \epsilon t) + e^{\beta \epsilon t},$$
            which shows that  $\lambda(\beta) < \beta $.  By definition,  $\lambda$ is nondecreasing. Due to Jensen's inequality, $\lambda$ is convex, and therefore continuous.   Letting $$\rho_\beta = 
    E_0 \left [ \exp (\beta \I(1)) \delta_0(X(1)) \right ],$$ it follows from monotone convergence that  $\rho _\beta>1$ for all sufficiently large $\beta$. Since $Z_{\beta,n} \ge \rho_\beta^n$, we conclude that $\lambda(\beta)>0$ for all $\beta$ large enough. We can therefore define a {\it critical value} of the parameter, $\beta_{cr}\in [0,\infty)$, by letting 
    $$\beta_{cr} = \sup  \{\beta: \lambda(\beta)=0\}.$$
    We summarize these findings in the following:
          \begin{prop}
          \label{lem:lambda_class}
 $\lambda$ is a nonnegative, nondecreasing  and convex function of $\beta$. Furthermore, there exists $\beta_{cr}=\beta_{cr}(d)\in [0,\infty)$ such that $\lambda(\beta)>0$ if and only if $\beta > \beta_{cr}$. 
   \end{prop} 
  We name  the parameter regimes $\beta<\beta_{cr}$, $\beta=\beta_{cr}$ and $\beta>\beta_{cr}$ 
        the {\it subcritical}, {\it critical} and {\it supercritical} phases, respectively. 
        
Using merely  Proposition  \ref{lem:lambda_class} and the large deviations statement on the reference measure \eqref{eq:LDI}, we are able to immediately prove the following:
\begin{prop}~
 \label{pr:homo_class}
 \begin{enumerate} 
 \item Suppose  $\beta< \beta_{cr}$. Then for any $\epsilon>0$, 
      $$\limsup_{t\to\infty} \frac{1}{t} \ln P_{\beta,t} (\I(t) \ge \epsilon t) <0;$$
 \item Suppose $\beta > \beta_{cr}$. Then for $\epsilon< \lambda(\beta)/\beta$      $$ \liminf_{t\to\infty} P_{\beta,t} (\I(t) \ge \epsilon t) >0.$$
\end{enumerate} 
\end{prop} 
Thus when $\beta<\beta_{cr}$ the qualitative nature of the large deviations statement for the reference measure \eqref{eq:LDI} is preserved under the homopolymer, while when $\beta>\beta_{cr}$ the homopolymer exhibits a completely opposite, positive-recurrent-like behavior. 
 \subsection{Organization} 
 The essential results on the model are presented in Section \ref{sec:results}, starting with a review of prior work in Section \ref{sec:CM}, and followed by our results on convergence of the homopolymer measure in Section \ref{sec:poly_conv}, and on convergence of the scaled homopolymer measure in Section \ref{sec:scaling}. An application of our results to a wetting model is given in Section \ref{sec:wetting}. The proofs of our results and many auxiliary results are given in Section \ref{sec:proofs}, with the exception of two lemmas  the authors consider as folklore, and whose proofs are given in the appendix for completeness. 
\section{Results} 
\label{sec:results}
\subsection{Previous work by Cranston and Molchanov}
\label{sec:CM}
 In  \cite{CM}, Cranston and Molchanov considered the problem of the limit behavior of $P_{\beta,t}$ as $t\to\infty$ when $\beta >0$. Their results fall into two different notions of limit. First is  the limit of $P_{\beta,t} |_{{\cal F}_T}$ as $t\to\infty$ while $T$ remains  fixed. The second is the limit of $X(t)$  (sometimes $X(t)/\sqrt{t}$, depending on $\beta$ and $d$)  under $P_{\beta,t}$ as $t\to\infty$. Here is a summary of the results in \cite{CM},  labeled here  as {\bf CM-x}, with x=I,II or III. 
 
 We need some definitions. First, let $\Delta$ denote the normalized discrete Laplacian on $\Z^d$. That is, for $u:\Z^d \to \R$ we have 
 \begin{equation} 
 \label{eq:laplacian} 
 \Delta  u (x) = \frac{1}{2d} \sum_{|y-x|=1} \left ( u(y)-u(x)\right),
 \end{equation}
  where here and henceforth $|\cdot|$ denotes the $\ell^2$-distance in $\C^d$.
  Recall that $\Delta$ is the generator of the simple symmetric nearest-neighbor random walk on $\Z^d$ with jump rate $1$ from each site.    Next, 
  define the perturbed operator  $H_\beta = \Delta + \beta \delta_0(x)$.  
  Let $\Sigma_\beta$ denote the spectrum of $H_\beta$ as an operator on $\ell^2(\Z^d)$. Since $H_\beta$ is bounded and self-adjoint, $\Sigma_\beta$ is a compact subset of $\R$. 
  \begin{CMth} 
We have
 $\lambda(\beta)= \max \Sigma_\beta$. In addition: 
  \begin{enumerate} 
 \item $\beta>\beta_{cr}$ if and only if $\lambda(\beta)$ is an isolated element in $\Sigma_\beta$. In this case $\lambda(\beta)$ is the only strictly positive element in $\Sigma_\beta$; and 
  \item $\beta_{cr}= P_0( X\mbox{ does not return to }0)= \sup \{\beta: \lim_{t\to\infty} Z_{\beta,t}<\infty\}$.
 \end{enumerate}  
\end{CMth} 
When $\beta> \beta_{cr}$, since $\lambda(\beta)$ is larger than any other element in the spectrum, the polymer can be easily analyzed through the spectral theorem, and up to leading order, the behavior is determined by the eigenfunction. When $\beta \le \beta_{cr}$, $\lambda(\beta)$ is not isolated (it is an eigenvalue only when $d\ge5$), and the analysis is  more delicate, except for the case $\beta< \beta_{cr}$ and $d\ge 3$, in which the  partition function converges to a limit in $(0,\infty)$ as $t\to\infty$. 
This suggests that the analysis of the polymer for $d\ge 3$ and $\beta < \beta_{cr}$ is simpler compared to the remaining phase  $\beta= \beta_{cr}$ for $d\ge 3$ and $\beta <0=\beta_{cr}$ for $d =1,2$. 

For the next result, we recall the notion of an  {\it $h$-transform}, also known as {\it Doob transform}. Let $H$ be a linear operator on a subspace  ${\mathbf V}$  of real-valued functions on $\Z^d$ and let $h:\Z^d \to (0,\infty)$. Let ${\mathbf V}^h$ denote the vector space of all functions $u:\Z^d \to \R$ such that $uh,\frac 1 h H(uh) \in {\mathbf V}$. Then we define the linear operator $H^h$,  the $h$-transform of  $H$, through 
 $$ H^h u = \frac{1}{h} H (uh),~u \in {\mathbf V}^h.$$
  \begin{CMth}
\label{th:CM_SR} 
  Let $\beta>0$. 
  \begin{enumerate}
  \item There exists a strictly positive function function $\psi_\beta$ on $\Z^d$  such that $\lim_{t\to\infty} \frac{Z_{\beta,t}(x)}{Z_{\beta,t}} = \psi_\beta (x)$.
    \item $P_{\beta,t}|_{{\cal F}_T}$  converges weakly as $t\to\infty$ to the distribution of the Markov process whose generator is  $(H_\beta-\lambda(\beta))^{\psi_\beta}$, conditioned on $X(0)=0$.  
  \end{enumerate} 
   \end{CMth}
  The discrete-time version of part ii in one dimension appears in \cite{GT}. 
  \begin{CMth}
  \label{th:CM_LR}
   Let $\beta>0$. Then the following hold: 
    \begin{enumerate} 
      \item Suppose  $d\ge 3$ and $\beta <\beta_{cr}$ then $\lim_{t\to\infty} P_{\beta,t} (X(t)/\sqrt{t} \in \cdot )  = N(0,\mbox{I}_d)$;  
      \item Suppose $d\ge 3$ and $\beta=\beta_{cr}$. 
       \begin{enumerate} 
        \item If $d=3,4$ then 
       $$\int_{{\mathbb R}^d}e^{i \phi \cdot x } P_{\beta,t} (\frac{X(t)}{\sqrt{t}}\in dx ) \Rightarrow \begin{cases}  \int_0^1 \frac{ e^{-\frac{ |\phi|^2}{6} (1-u) }du}{2\sqrt{u}}du & d=3 \\ \int_0^1 e^{\frac{-|\phi|^2}{8} u}du & d=4; \end{cases}$$
         \item If $d\ge 5$ then 
         $\lim_{t\to\infty} P_{\beta,t}(X(t) =x ) = \frac{\psi_\beta (x)}{\sum_y \psi_\beta (y)}$. 
         \end{enumerate} 
    \item If $\beta > \beta_{cr} $ then $\lim_{t\to\infty} P_{\beta,t}(X(t) =x ) = \frac{\psi_\beta (x)}{\sum_y \psi_\beta (y)}$. 
    \end{enumerate} 
    \end{CMth} 
 The limit as a mixture of normals in part ii-(a) is from \cite[Theorem 2.1]{CM2}. The variance here differs from that paper,  due to our choice to work with the normalized  Laplacian \eqref{eq:laplacian}, in accordance with \cite{CM}.  
    \subsection{Convergence of Polymer}
    \label{sec:poly_conv}
    The main goal in this section is to understand the distribution of $X$, restricted to some fixed interval $[0,T]$ under $P_{\beta,t}$ as $t\to\infty$. We will show that as in Theorem \ref{th:CM_SR}, the polymer converges to a Doob transform. Our model is essentially a ``soft" version of random walk conditioned not to hit the origin, and a discrete version of penalized Brownian motion, for which similar results have been proved. 
  
    We begin with a general discussion motivated by  Theorem \ref{th:CM_SR}. Let $\beta \in \R$ and consider the problem  
  \begin{equation} 
  \label{eq:prob} 
  \begin{cases} 
  (H_\beta -\lambda(\beta)) u =0;\\
   u \ge 0
   \end{cases} 
   \end{equation}
   Define the hitting time $\tau_0$, 
   $$ \tau_0=\inf\{t\ge 0: X(t) = 0\}. $$
   For $\varphi= (\varphi_1,\dots,\varphi_d) \in \R^d$, let $\Phi(\varphi)=\frac {1}{d}\sum_{j=1}^d (1-\cos \varphi_j)$, and let 
   $$ A_0(x) = \frac{1}{\pi^d} \int_{[0,\pi]^d} \frac{ 1- \cos (\scp{\varphi,x})}{\Phi(\varphi)} d \varphi.$$ 
   We have the following result:
 \begin{theorem}
\label{th:psib_desc}
 \begin{enumerate}~ 
 \item Suppose $\beta > \beta_{cr}$. Then the cone of bounded solutions to \eqref{eq:prob}  is spanned by the function $E_x e^{-\lambda(\beta) \tau_0}$.
\item Suppose  $\beta \le \beta_{cr}$. Then the cone of solutions to  \eqref{eq:prob} 
is spanned by $u$, where 
$$u(x)=\begin{cases} 1- \beta A_0(x)&d=1,2;\\
 1- \frac{\beta}{\beta_{cr}}P_x (\tau_0=\infty)& d\ge 3.\end{cases}$$
\end{enumerate} 
 \end{theorem}
 The solution to \eqref{eq:prob} given in Theorem \ref{th:psib_desc} will be henceforth denoted by $\psi_\beta$, in agreement with the notation in Theorem \ref{th:CM_SR}. We also comment  that when $d=1$, $A_0(x)=|x|$. 
 \begin{theorem}
  \label{th:TR_theorem}
  The  Markov chain generated by $(H_\beta - \lambda(\beta) )^{\psi_\beta}$ is recurrent if and only if $\beta \ge \beta_{cr}$. 
 It  is positive recurrent if and only if $\beta>\beta_{cr}$ or $\beta =\beta_{cr}$ and $d \ge 5$. 
 \end{theorem} 
 We now focus on the subcritical phase $\beta <\beta_{cr} =0$ for $d=1,2$.  Note that when  $d\ge3$, the transience of $X$ under the reference measure implies that the total time spent at $0$  is $\mbox{Exp}(\beta_{cr})$,  because it is the sum of a $\mbox{Geom}(\beta_{cr})$-distributed number of independent $\mbox{Exp} (1)$-distributed random variables, each exponential random variable representing the duration of a single visit to $0$. As a result, when $d\ge 3$ and $\beta<\beta_{cr}$ we have $\lim_{t\to\infty} Z_{\beta,t} = \frac{\beta_{cr}}{\beta_{cr}-\beta} \in (0,\infty)$. The analysis carried out in \cite{CM} for $d\ge 3,~\beta \in (0,\beta_{cr})$ rests only on the fact that in this phase  $\lim_{t\to\infty}Z_{\beta,t} \in (0,\infty)$ and therefore extends seamlessly to $\beta< \beta_{cr}$.  
    
What makes the parameter regime $d=1,2$, $\beta<0$, interesting is the following.  Firstly, this is the only parameter regime for which  $\lim_{t\to\infty} Z_{\beta,t}=0$ (but not exponentially). Secondly, it exhibits an interplay between recurrence for the reference measure, working in favor of returning to $0$, versus the negative parameter,  which penalizes  staying at $0$. In spirit, this regime resembles the critical phase for $d=3,4$, with some extra care required due to recurrence which causes some integrals (resolvents) to blowup. 
 
Let $Z_{\beta,t} (x)= E_x [e^{\beta\I(t)}]$. Note that   $Z_{\beta,t}=Z_{\beta,t}(0)$. We have: 
        \begin{theorem}
        \label{th:Zbeta}
         Let $\beta<0$. Then 
         \begin{enumerate}
           \item
   \label{Zbeta_cor}
    $Z_{\beta,t} \sim \begin{cases}-\frac{1}{\beta} \sqrt{\frac{2}{\pi t}} & d=1;\\  -\frac{\pi}{\beta \ln t} & d=2.
       \end{cases}$   
          \item 
          \label{Zbeta_martin}
           $Z_{\beta,t}(x)\sim  \psi_\beta(x)Z_{\beta,t}$.
               \end{enumerate}  
           \end{theorem}
           Combining this theorem with a simple tightness argument leads to the proof of the following extension of Theorem \ref{th:CM_SR}: 
          \begin{theorem}
          \label{th:small times}
          Let $\beta<0$ and let $T>0$. Then  $P_{\beta,t}|_{{\cal F}_T}$  converges weakly as $t\to\infty$ to the distribution of the Markov process whose generator is  $(H_\beta)^{\psi_\beta}$, conditioned on $X(0)=0$.  The transition  function for this process, $q_\beta$, is given by 
          $$ q_\beta (t,x,y) = \frac{1}{\psi_\beta (x)} E_x\left [ e^{\beta \I(t)} \psi_\beta (X(t)) \delta_y (X(t))\right ] .$$ 
           \end{theorem} 
           In what follows we will denote the distribution of the process generated by $H_{\beta}^{\psi_{\beta}}$ by $Q$.  As before, $Q_x$ will denote the distribution of the process starting from $x$, and $E^Q_x$ will denote the expectation of the process starting from $x$.  Observe that there's a tight relation between $Q$ and $P_{\beta,t}$. 
           Indeed, if $h_1,\dots, h_k$ are continuous real-valued bounded functions on $\R$ an $0<t_1<\dots< t_n \le t$,  then from the definition of the transition kernel $q_\beta$, we have that 
           $$ \int \prod_{j=1}^n h_j (X_{t_j} ) d P_{\beta,t} = \frac{1}{\psi_\beta(0)}E^Q_0[ \prod_{j=1}^n h_j (X_{t_j}) \frac{1}{\psi_\beta (X_t)}].$$ 
           In particular,  (and since  $\psi_{\beta}(0)=1$ from its definition), 
           \begin{align} \label{eq:Q_pol}  d P_{\beta,t} &=  \frac{  \frac{1}{\psi_{\beta} (X_t)} d Q_0|_{{\cal F}_t}}{Z_{\beta,t}},\mbox{ and } \\
           \nonumber 
         Z_{\beta,t} &= E^Q_0 [ \frac{1}{\psi_{\beta}(X_t)}],
         \end{align} 
 
                             Let $\sigma_t = \sup\{s\le t:X(s)=0\}$ and 
            $N_t=\# \{s\le t : X(s)=0\mbox{ and } X(s^-)\ne 0\}$.  We have the following corollary to Theorem \ref{th:Zbeta} and Theorem \ref{th:small times}:
                \begin{corollary} 
              Let $\beta <0$. Then,
              \label{cor:lasttimes} 
         \begin{enumerate}
         \item $ P_{\beta,t} (\I(t) \in \cdot ) \Rightarrow   Exp (-\beta)$;
         \item $ P_{\beta,t} (\sigma_t\in \cdot) \Rightarrow -\beta p_{\beta}(y,0,0) dy$; and 
         \item $ P_{\beta,t}(N_t\in \cdot) \Rightarrow Geom(\frac{-\beta}{1-\beta})$. 
      \end{enumerate}
      \end{corollary}
      We comment that as is easy to verify, the limiting distributions above coincide with the respective distributions of $\lim_{t\to\infty} \I(t)$, $\lim_{t\to\infty} \sigma_t$ and $\lim_{t\to\infty} N_t$ under $Q_0$. 
          \subsection{Convergence of the scaled polymer} 
          \label{sec:scaling}
          In this section we will consider the behavior of the polymer when it is space- and time-scaled. To this end, let us introduce the scaled polymer. For $n\in\N$, 
          let $X^{(n)}$ denote the process defined by $X^{(n)}_t = X_{nt}/\sqrt{n},~t \in [0,1]$. This is the rescaled process. Our main goal is to obtain a functional central limit theorem for $X^{(n)}$.   As the polymer is a discrete analog of the penalized Brownian motion of \cite[Theorem 4.16, p. 251]{roy_yor}, it is not surprising that the scaling limits obtained coincide with those for the penalized Brownian motion. In fact, the only difficulty in the proof is in showing that the discrete process does converge to its continuous counterpart. What makes this convergence non trivial is the lack of stochastic analysis, scaling invariance, and the fact that the limit processes involve diffusion with singular coefficients, that is Bessel-3 process. We study the model through analysis of resolvents. We comment that the model is also amenable to analysis through the powerful renewal approach presented in  \cite{giacomin},  and more specifically in \cite{CGZ}.  
                      
          We will introduce some notation. The Brownian meander is defined as follows. Let $W=\{W_t:t\ge 0\}$ be standard $1$-dimensional Brownian motion, and let $L=\sup\{t\le 1:W_t=0\}$. Then for $t \in [0,1]$, let $M_t = \frac{ |W_{(1-t) L + t}|  }{\sqrt{1-L}}$. The resulting process $M=\{M_t:t\in [0,1]\}$ is called the Brownian meander. 
          Recall that the Bessel-3 process, which we denote by $R=\{R_t:t\ge 0\}$ is the Markov process on $[0,\infty)$ generated by 
           $$ \frac{1}{2} \frac{d^2}{dx^2} + \frac{1}{x} \frac{d}{dx},$$
           Unless otherwise specified, we will  assume that $R_0=0$. 
            If $W=\{W_t:t\ge 0\}$ standard  Brownian motion on $\R^3$, then $|W|=\{|W_t|:t\ge 0\}$ has the same distribution as $R$.
            
             The Brownian meander and the Bessel-3 process are related through the Imhof relations which state that for every  bounded continuous function $F:D[0,1]\to \R$, we have
          \begin{equation*} 
           E [ F (M)  ] = \sqrt{\frac{\pi}{2}} E [ F (R)  \frac{1}{R_1}],~\sqrt{\frac{2}{\pi}} E [ F (R) ] = E [ F(M) \frac{1}{M_1}].
          \end{equation*}
          In particular, it follows that 
          \begin{equation}
          \label{eq:imhof}
            E [ F (M)  ]  = \frac{E [ F (R)  \frac{1}{R_1}]}{E [\frac{1}{R_1}]}.
          \end{equation}  
          From the Imhof relation, and using the fact that $R_1$ has density $$\sqrt{\frac{2}{\pi}} \ch_{[0,\infty)} (x) x e^{-x^2/2},$$  we conclude that 
          \begin{equation} 
          \label{eq:imhof_little} 
          P (M_1 \in \cdot ) =  \ch_{[0,\infty)}(x) x e^{-x^2/2} dx .
          \end{equation}  
   
          In the results below, $J$ denotes a Bernoulli random variable with $P(J=1)=P(J=-1)=\frac 12$, independent of $R$ and $M$. 
 
          \begin{theorem}
          \label{th:scaling_limit} 
         Suppose $d=1$ and $\beta <0$.  Then 
         \begin{enumerate} 
         \item
         $\displaystyle Q_0( X^{(n)}\in \cdot ) \Rightarrow J R$, and 
       \item 
        $\displaystyle P_{\beta,n} (X^{(n)} \in \cdot ) \Rightarrow J M$
        \end{enumerate} 
         \end{theorem} 
         
          An immediate consequence of this theorem is that $ Q_0 ( \frac{X_n}{\sqrt{n}}  \in \cdot ) \Rightarrow JR_1$ and $ P_{\beta,n} ( \frac{X_n}{\sqrt{n}}  \in \cdot ) \Rightarrow JM_1$.
          
    We continue to a short informal discussion of the case $d=2$, explaining  why, in our opinion, it is less interesting. The bottom line is  the penalizing does not affect  the scaling limit all.  First, observe that from \eqref{eq:Q_pol}, 
      $$ \int F (X^{(n)}) d P_{\beta,t}  = \frac{ E^Q_0 [ F (X^{(n)}) \frac{1}{\psi_{\beta} (X_n) }]}{E^Q_0 [\frac{1}{\psi_{\beta} (X_n)}]}.$$ 
      It could be shown that $\psi_{\beta}$ grows logarithmically. Therefore, the righthand side is asymptotically equivalent to 
      $$  \frac{ E^Q_0 [ F (X^{(n)}) \frac{1}{\ln( \sqrt{n} |X^{(n)}_1|)}]} {E^Q_0 [ \frac{1}{\ln( \sqrt{n} |X^{(n)}_1|)}]} \sim E^Q_0 [ F(X^{(n)}) ],$$ 
      if under $Q_0$, $X^{(n)}_1=O(1)$.  The  logarithmic growth of $\psi_\beta$ and  the arguments of Section \ref{sec:martigale_prob} guarantee this is indeed the case, and, in  addition, that the generator of $X^{(n)}$ converges to $\frac{1}{2} \Delta$ on $\R^2$. The latter statement, along with a  tightness argument adapted from Section \ref{sec:tightness} to this setting, then imply that the law of $X^{(n)}$ under $Q_0$ converges to Brownian motion in two dimensions starting from the origin and it then follows that the polymer has the same limit. 
 \subsection{Relation to a Wetting Model} 
 \label{sec:wetting}
 In this short section we consider a modified version of the polymer known in the literature as a wetting model \cite{IsYosh}\cite[Sec 1.3]{giacomin}. Our goal is to show how  results on the wetting model follow from our results on the polymer. For $\beta' \in \R$, and $t\ge 0$, let $\tilde P_{\beta',t}$ be the polymer measure on ${\cal F}_t$ defined as 
  $$ \frac{ d \tilde P_{\beta',t}}{d P|_{{\cal F}_t}} =  \frac{ e^{\beta' \I (t)} \ch_{A_t}}{\tilde Z_{\beta',t}},$$
  where $A_t = \{\inf_{s\le t}  X_s\ge 0\}$, and $\tilde Z_{\beta,t}$ is a normalizing constant.  The following is an immediate consequence of our analysis of the polymer model. 
  \begin{theorem}
   $\tilde P_{\beta',n} (X^{(n)} \in \cdot ) \Rightarrow  \begin{cases} M & \beta' < \frac 12 \\ |W| & \beta' = \frac 12.\end{cases}$ 
  \end{theorem} 
 We comment that the  discrete-time version of the theorem is \cite[Th. 1.2]{IsYosh}. \\
  To prove the theorem (and also understand the case $\beta' > \frac 12$), consider the transition kernel 
  $$\tilde p(t,x,y) = E_x [e^{\beta' \I (t)} \ch_{A_t}\delta_y(X_t)],~t\ge 0, x,y \in \Z.$$
  This defines a  semigroup on $\Z_+$ whose generator $\tilde H_{\beta'}$ is the restriction of $H_{\beta'}$ to functions vanishing on $\{-1,-2,\dots\}$. 
   Observe  that 
  $$ (\tilde H_{\beta'} f)(x)  =\begin{cases}  \frac 12 f (1) +(\beta' - 1) f(0) & x=0 \\ \Delta f (x) & x>0.\end{cases}$$
  In particular, $\tilde H_{\beta'}$ is the generator of a Markov process if and only if $\beta' =\frac 12$, and in this case it is the generator of random walk reflected at the origin. It follows that $\tilde P_{\frac 12,n}(X^{(n)} \in \cdot) \Rightarrow |W|$ on $D[0,1]$ as $n\to\infty$. For $\beta' > \frac 12$, the analysis is identical to the supercritical phase for the hompolymer, that is, the principal eigenvalue for $\tilde H_{\beta'}$ is an isolated eigenvalue with an eigenfunction in $\ell^2 (\Z_+)$, and the corresponding results hold verbatim. As for $\beta' < \frac 12$, if $\phi$ is a positive harmonic function for $\tilde H_{\beta'}$, then without loss of generality $\phi(0)=1$ and as a result $\phi(1) =2(1-\beta')=1+ (1-2\beta')$. Since $\tilde H_{\beta'}$ coincides with $\Delta$ on $\N$, this implies that for $x\in\Z_+$, $\phi(x) = 1+ (1-2\beta') x$, that is,  $\phi$ coincides with the restriction of $\psi_{2\beta'-1}$ to $\Z_+$, where $\psi_{\cdot}$ is the function  from Theorem \ref{th:psib_desc}-(ii) for $d=1$. Hence the Markov process generated by the $h$-transformed  $(\tilde H_{\beta'})^{\phi}$ coincides $|X|$ under $Q$ with parameter $\beta=  2\beta'-1$. In particular, for any bounded and  continuous $F:D[0,1]\to \R$, 
  $$ E [ F(X^{(n)} )e^{\beta' \I (n) } \ch_{A_n} ] = E^{Q}_0 [ \frac{ F(|X^{(n)}|)}{\psi_{\beta}(|X_n|)}].$$
It follows from \eqref{eq:Q_pol}
$$ \tilde P_{\beta',n} (X^{(n)}  \in \cdot) = P_{\beta,n} (|X^{(n)}| \in \cdot).$$ 
and from  Theorem \ref{th:scaling_limit}-(ii) we have that $\tilde P_{\beta',n} (X^{(n)} \in \cdot ) \Rightarrow M$.  
\section{Proofs}
\label{sec:proofs}
\subsection{Preliminaries} 
 By the Feynman-Kac formula, $H_\beta$ generates a semigroup whose transition function $p_\beta(t,x,y)$ is given by 
            \begin{equation}
            \label{eq:FC} 
            p_{\beta}(t,x,y)=E_x  \left [ e^{\beta \I(t)}\delta_y(X(t))\right].
            \end{equation}
            Recall that $\Sigma_\beta$ is  the spectrum of $H_\beta= \Delta + \beta \delta_0(x)$ on $\ell^2(\Z^d)$.  
                   For $\lambda \not \in \Sigma_\beta$, we define the resolvent $R_\lambda^\beta$ as $R_{\lambda}^{\beta}=(\lambda-H_\beta)^{-1}$. Then    $$ R_{\lambda}^{\beta} f (x) = \int_0^{\infty} e^{-\lambda t} \sum_y p_{\beta}(t,x,y)f(y)dt=\int_0^{\infty} e^{-\lambda t} E_x \left [f(X(t)) e^{\beta \I(t)}\right ]dt.$$
Abusing notation, we write $R_{\lambda}^{\beta}(x,y)$ for $R_{\lambda}^{\beta}\delta_y(x)$. 
  We also write $R_\lambda$ for $R_\lambda^0$. 
   Due to the special role of $R_{\lambda}^{\beta}(0,0)$ we denote it by $I^\beta(\lambda)$, and write $I$ for $I^0$.  The resolvent $R_\lambda^\beta$ can be obtained directly from $R_\lambda$ through the resolvent equation, as we now show. Suppose that  $\lambda \not \in \Sigma_0\cup \Sigma_\beta$. Then 
  $(\lambda- H) R_\lambda^\beta(x,y) = \delta_y(x)+\beta R_\lambda^\beta  (0,y)\delta_0(x)$. 
   Hence
    $$R_\lambda^\beta (x,y) = R_\lambda(x,y)+\beta R_\lambda^\beta (0,y) R_\lambda(x,0),$$
     and so by letting $x=0$, we obtain 
     \begin{equation}
     \label{eq:prelres} 
      R_\lambda^\beta(0,y) = \frac{R_\lambda(0,y)}{1-\beta I(\lambda)},
      \end{equation}
       which gives 
      \begin{equation} 
      \label{eq:resolventeqn}
      R_\lambda^\beta(x,y)=R_\lambda(x,y)+\frac{\beta R_\lambda(x,0)R_\lambda(0,y)}{1-\beta I(\lambda)}.
      \end{equation}
       In particular, 
      \begin{equation}
      \label{eq:Ilambda} 
       I^\beta (\lambda)=\frac{I(\lambda)}{1-\beta I(\lambda) }= \frac{-\beta^{-1}}{1- \frac{1}{\beta I(\lambda)}},
        \end{equation} 
By \eqref{eq:FC} 
 $Z_{\beta,t}(x) = p_{\beta}(t,x,{\bf 1})=\sum_{y\in \Z^d} p_\beta(t,x,y)$. Also,  
             $$\frac{d Z_{\beta,t}}{dt} = \beta E_0  [e^{\beta\I(t)}\delta_0(X(t))] = \beta p_\beta(t,0,0).$$
          Thus, 
           $$ Z_{\beta,t} = 1+ \beta \int_0^t p_\beta(s,0,0)ds.$$
      Assume that  $d=1,2$ and $\beta<0$. Then $\lim_{t\to\infty} Z_{\beta,t}=0$, and therefore 
                      \begin{equation}
             \label{eq:Zbeta}
              Z_{\beta,t} = -\beta \int_t^{\infty} p_\beta(s,0,0) ds. 
              \end{equation} 
   We now  derive an integral  representation for $R_\lambda$ through Fourier transforms. By reversibility and spatial homogeneity, $R_\lambda(x,y)=R_\lambda(y,x)=R_\lambda(x-y,0)$.   The Fourier series  of  $R_\lambda(\cdot,0)$,  denoted by   $\hat R_{\lambda}$, is defined through
   $$\hat R_{\lambda}(\varphi) =\sum_{x\in \Z^d } R_{\lambda}(x,0)e^{i\scp{\varphi,x}},~\varphi \in [0,2\pi]^d.$$ 
   Since $\{\frac{e^{i \scp{\varphi,x}}}{(2\pi)^{d/2}}:x \in \Z^d\}$ form an orthonormal basis for $L^2([0,2\pi]^d)$, the inversion formula is given by 
         $$ R_{\lambda}(x,0)=\frac{1}{(2\pi)^d}\int_{[0,2\pi]^d} \hat R_{\lambda}(\varphi) e^{-i\scp{\varphi,x}}d\varphi.$$
  Taking the Fourier series of both sides of  $(\lambda-\Delta ) R_{\lambda}(x,0)=\delta_0(x)$,  it follows that $(\lambda +\Phi(\varphi))\hat R_{\lambda}(\varphi)=1$, where $\Phi(\varphi)=\frac {1}{d}\sum_{j=1}^d (1-\cos \varphi_j)$ is the symbol of $-\Delta$. Therefore
   \begin{align}
  \nonumber
   R_\lambda(x,y)&= \frac{1}{(2\pi)^d}\int_{[0,2\pi]^d}\frac{e^{-i\scp{\varphi,x-y}} }{\lambda+\Phi(\varphi)}d\varphi\\
   \label{eq:RlambdaFourier}
  &=  \frac{1}{\pi^d}\int_{[0,\pi]^d}\frac{\cos (\scp{\varphi,x-y}) }{\lambda+\Phi(\varphi)}d\varphi.
  \end{align}
   The second equality is due to the facts that the integrand is symmetric about $\pi$ with respect to each of the variables and that $R_\lambda$ is real for real $\lambda$. Setting $x=y=0$ in \eqref{eq:RlambdaFourier}, we obtain 
     $$I(\lambda) = \frac {1}{\pi^d} \int_{[0,\pi]^d}\frac{1}{\lambda+\Phi(\varphi)}d\varphi, $$
    and an estimate of this integral leads to the following known lemma. The proof is given in the Appendix. 
         \begin{lemma}
 \label{lem:Ilambda}
 $$ I(\lambda) = \begin{cases}  \frac{1}{\sqrt{\lambda(2+\lambda)}}& d=1;\\ 
    \left(- \frac{\ln |\lambda|}{\pi} +O_{\mbox{real}}(1)\right)-i (1+ o_{\mbox{real}}(1))  \mbox { as } \lambda \to 0 & d=2,\end{cases}$$
    Where the subscript $\mbox{real}$ means that the function is real-valued.  
\end{lemma}
The next result is obtained by inverting the Laplace transforms in Lemma \ref{lem:Ilambda}.
\begin{prop}
\label{pr:pbeta00}
\begin{equation*}
\label{eq:pbeta} p_{\beta}(t,0,0) \sim \begin{cases}  \frac{1}{\sqrt{2\pi} \beta^2 t^{3/2}} & d=1;\\
          \frac{\pi}{\beta^2 t (\ln t)^2} & d=2. \end{cases}
 \end{equation*}
\end{prop} 
\begin{proof} 
By the Spectral Theorem, there exists a probability measure $\mu_\beta$ on  $\Sigma_\beta$, such that 
        \begin{equation} \label{eq:spectral_lambda}  I^\beta (\lambda)=\int_{\Sigma_\beta} \frac{d\mu_\beta(s)}{\lambda-s}.\end{equation} 
        and 
        $$ p_\beta(t,0,0) = \int_{\Sigma_\beta} e^{ts} d\mu_\beta(s).$$ 
        Since $p_\beta(t,0,0)\to 0$ as $t\to\infty$ and $\Sigma_\beta \subset (-\infty,0]$, it follows that $\mu_\beta (\{0\}) =0$. In addition, from \eqref{eq:spectral_lambda}, we observe that for $\lambda = s_0+ i\epsilon$ with $\epsilon>0$, $\lambda -s = (s-s_0)+ i\epsilon$, therefore $\Im(1/ (\lambda-s)) = -\frac{\epsilon}{(s-s_0)^2+\epsilon^2}$, and so, $-\frac 1\pi \Im (1/(\lambda -s))$ is an approximation of the identity, and it follows that $h_\beta$, the density of the absolutely continuous part of $\mu_\beta$ is given by the formula 
       $$h_\beta(s)  = \lim_{\epsilon \searrow  0} - \frac{1}{\pi}\Im I^\beta(s+i\epsilon),$$
       and that the singular part of $\mu_\beta$ is supported on $\{s:\limsup_{\epsilon \searrow  0} \Im I^\beta(s+i\epsilon)=\infty\}$.  By Lemma \ref{lem:Ilambda} $\lim_{\lambda \to 0}| I(\lambda)|=\infty$, it follows from \eqref{eq:Ilambda} that $I^\beta$ is bounded near the origin, which then guarantees that $\mu_\beta$ is absolutely continuous near the origin. As a result, there exists some $\delta>0$ such that 
 \begin{equation}
 \label{eq:Ilambdaspec}  p_\beta(t,0,0) = \int_{\Sigma_\beta} e^{ts} d\mu_\beta(s) \sim  \int_{0}^{\delta} e^{-t u} h_\beta(-u) du + O(e^{-\delta t})~\mbox{ as } t\to\infty.
 \end{equation} 
When $d=1$, $\lim_{\epsilon \searrow 0} I (s+i\epsilon)=\frac{e^{-i \pi/2}}{ \sqrt{|s|}\sqrt{2+s}}$. This gives  $h_\beta (s)\sim  \frac{ \sqrt{2|s|}}{\beta^2 \pi}\mbox{ as }s\nearrow 0$. Therefore by \eqref{eq:Ilambdaspec} we obtain
 \begin{align*}
   p_\beta(t,0,0)&\sim \frac{\sqrt{2}}{\beta^2 \pi }\int_0^\delta e^{-tu} \sqrt{u}du +O(e^{-\delta t})\\
   &= 
    \frac{\sqrt{2}}{\beta^2 \pi }\int_0^{\delta t}  e^{-v} \sqrt{v/t}d(v/t) +O(e^{-\delta t})\sim \frac{\sqrt{2}}{\beta^2 \pi t^{3/2}}\int_0^\infty e^{-v}  \sqrt{v} dv \\
    &= \frac{1}{\beta^2 \sqrt{2\pi} t^{3/2}},~\mbox{ as }t\to\infty.
    \end{align*}
    Suppose now $d=2$ and $\lambda = -s+\epsilon$ for some $s\ge 0$ and $\epsilon>0$.  It follows from the Lemma \ref{lem:Ilambda} that 
    $$ I^\beta (\lambda) =\frac{I(\lambda)(1- \beta \ol{I(\lambda)})}{|1- \beta I (\lambda)|^2}.$$
     Therefore, 
     $$- \frac{1}{\pi} \Im I(\lambda) = \frac {1}{\pi} \frac{-\Im I(\lambda) }{|1- \beta I(\lambda)|^2}\sim \frac{\pi }{\beta^2 (\ln |s|)^2/\pi^2},$$
          so that $\mu_\beta$ is absolutely continuous on  an interval $(-\delta,0)$ and we have  $h_\beta(s) \sim \frac{\pi}{\beta^2 (\ln |s|)^2}$. Furthermore, since $p_\beta(t,0,0)\to 0$, it is clear that $0$ is not an atom for $\mu_\beta$. Therefore, it follows from   \eqref{eq:Ilambdaspec} that  
 \begin{align*}
   p_\beta(t,0,0)&\sim \frac{\pi }{\beta^2 }\int_0^\delta e^{-tu} \frac{1} {(\ln u)^2} +O(e^{-\delta t})\\
   &= 
    \frac{\pi }{\beta^2 }\int_0^{\delta t}  e^{-v}\frac{1}{(\ln (v/t))^2} d(v/t)  +O(e^{-\delta t})\\
    &\sim \frac{\pi}{\beta^2  t (\ln t)^2 }\int_0^{\delta t} e^{-v}  \frac{1}{\left(\frac{\ln v}{\ln t}-1\right)^2} dv \sim  \frac{\pi }{\beta^2 t (\ln t)^2},~\mbox{ as }t\to\infty.
    \end{align*}
    \end{proof}
  Next we  study the solutions to \eqref{eq:prob}  through Martin boundary theory. For this we consider the simple symmetric random walk on $\Z^d$ killed upon hitting $0$.   The generator of this sub-markovian process is the restriction of $\Delta$ to functions vanishing at $0$, which could be formally written as  $H_{-\infty}$. Let $R_\lambda^{-\infty}= \lim_{\beta \to \infty} R_\lambda^{-\beta}$. Then  \eqref{eq:resolventeqn} shows that
   \begin{equation} 
   \label{eq:killed_res}  R_\lambda^{-\infty}(x,y) = R_\lambda(x,y) - \frac{R_\lambda (x,0)R_\lambda(0,y)}{I(\lambda)},
   \end{equation} 
    and an immediate calculation shows it is indeed the resolvent of the killed walk. In light of this discussion we identify the generator of the killed walk with $H_{-\infty}$. We need the following result: 
    \begin{lemma}
    \label{lem:Rlambda} 
    Let $B= \{e\in \Z^d:|e|=1\}$. Then for $x\ne 0$, 
      $$R_0^{-\infty}(x,\ch_B)=  2d   P_x (\tau_0<\infty).$$
    \end{lemma} 
    \begin{proof}
    We first observe that for all $x\ne 0$, 
  $$
    R_0^{-\infty} (x, \ch_B) = P_x (\tau_B <\infty) R_0^{-\infty} (e_1,\ch_B).
$$
   Let $q_B = P_{e_1} (\tau_B = \infty)$. The number of visits to $B$ starting from $e_1=(1,0,\dots)\in \Z^d$, until the first visit to the origin is clearly $Geom(q_B + \frac 1{2d})$. Thus, the time at $B$ starting from $e_1$, is the sum of $Geom(q_B+\frac {1}{2d})$ independent $Exp(1)$ random variables (and is therefore $Exp (q_B + \frac{1}{2d})$). As a result, $R_0^{-\infty} (e_1,B) = \frac{1}{q_B + \frac{1}{2d}}$. Next, observe that  $p_0$, the  probability that starting from the origin $X$ will return to $0$, is  equal to  $P_{e_1} (\tau_0 < \infty)$, and this is equal to $ \frac 1{2d} + (1-q_B - \frac {1}{2d}) p_0$.  Therefore, $2d p_0 = \frac{1}{q_B+ \frac{1}{2d}}$, so  that $R_0^{-\infty}(e_1,B) = 2d p_0$.      In addition, for any $x\in \Z^d$, $P_x (\tau_0 < \infty) = P_x (\tau_B < \infty) p_0$. Combining the two identities, the result follows. 
   \end{proof}

In order to study the martin boundary of $H_{-\infty}$ we define the function $A_\lambda$ as 
  \begin{equation}
  \label{eq:Alambda}
  A_\lambda(w) = I(\lambda)-R_\lambda(w,0)=  \frac{1}{\pi^d}\int_{[0,\pi]^d}\frac{1-\cos (\scp{\varphi,w}) }{\lambda+\Phi(\varphi)}d\varphi.
  \end{equation}
  We observe that $A_\lambda \in [0,\infty)$ and that  $A_\lambda(w)=0$ if and only if $w=0$. In addition, by dominated convergence, $A_0(w) =\lim_{\lambda \searrow 0} A_\lambda(w)$ exists and is finite.    We have the following known result, whose proof is given in the appendix.  
         \begin{lemma}~
        \label{lem:martin} 
        \begin{enumerate} 
        \item The  Martin boundary for $H_{-\infty}$ is spanned by $A_0(x)$.
        \item $A_0(e_1)=1$ and when  $d\ge3$, $A_0(x)=\frac{P_x (\tau_0 = \infty)}{\beta_{cr}}$.  
        \end{enumerate} 
         \end{lemma} 
         \subsection{Proof of Proposition \ref{pr:homo_class}}
         \begin{proof}
           When $\beta<\beta_{cr}$, we can choose $\alpha>0$  such that $\beta (1+\alpha)<\beta_{cr}$ and it follows from H\"older that
        $$ P_{\beta,t} (\I(t) \ge \epsilon t) \le P_0(\I(t) \ge \epsilon t)^{\alpha/(1+\alpha)}\frac{(Z_{(1+\alpha)\beta,t})^{1/(1+\alpha)}}{Z_{\beta,t}},$$
         and the first assertion follows by taking logarithms and then letting $t\to\infty$. 
       Next, when  $\beta > \beta_{cr}$ we have 
      $$\limsup_{t\to\infty} \frac{1}{t} \ln P_{\beta,t} (\I(t) \le \epsilon t)  \le \beta \epsilon - \lambda(\beta),$$
       the right-hand side being strictly negative, provided that $\epsilon<\lambda(\beta)/\beta$,  immediately leading  to the second assertion. 
    \end{proof}   
    
             \subsection{Proof of Theorem \ref{th:psib_desc}}
  \begin{proof} 
    We begin by proving existence.  Clearly, for all $\lambda> \lambda(\beta)$, the resolvent operator $R_\lambda^\beta$ defined in \eqref{eq:resolventeqn} is a bounded positive operator on $\ell^2(\Z^d)$.
    Existence of the positive resolvent automatically implies the existence of positive harmonic functions for $H_\beta -\lambda$.  Let $u_\lambda$ denote such a function satisfying $u_\lambda (0)=1$. 
  Now let $\lambda_n \searrow \lambda(\beta)$. It follows from Harnack inequality that by possibly passing to a subsequence $(u_{\lambda_n}:n\in \N)$ converges pointwise to a nonnegative harmonic function for $H_\beta - \lambda(\beta)$, $u_{\lambda(\beta)}$ satisfying $u_{\lambda(\beta)}(0)=1$.
 
  As for uniqueness, we first show that all  solutions which vanish at $0$ (or in fact at any other site) are identically zero. Suppose that $u$ is a solution satisfying $u(0)=0$, then since $H_\beta u (0) = \lambda(\beta) u(0)=0$, it follows that $\frac{1}{2d} \sum_{|e|=1} u(e)=0$, equivalently $u(e)=0$ for all $|e|=1$. By  induction on $|x|$, it  follows that $u \equiv0$.  Hence, in order to prove uniqueness, it is sufficient to show that there exists a unique solution $u$ with $u(0)=1$. 

  We split the proof of uniqueness according to the value of $\beta$. 
  \subsubsection*{Supercritical phase,  $\beta>\beta_{cr}$}
   Suppose that $u$ is a bounded solution to \eqref{eq:prob} with $u(0)=1$. Then  $(e^{-\lambda(\beta) t}u(X(t)):t\in \R_+)$ is a bounded ${\cal F}_t$-martingale with respect to the reference measure. In particular, $u(x) = E_x e^{-\lambda(\beta) \tau_0} u(X(\tau_0))= E_x e^{-\lambda(\beta) \tau_0}$. This proves uniqueness. 
      
As a side remark, we observe that  this allows to obtain   $\beta_{cr}$ directly, because on the one hand, $\lim_{\beta \searrow \beta_{cr}} E_{e_1}  e^{-\lambda(\beta) \tau_0}=P_{e_1} (\tau_0 < \infty)$ while on the other hand, $0= (H_\beta E_x e^{-\lambda(\beta) \tau_0})(0) = E_{e_1} e^{-\lambda(\beta) \tau_0}-(1-\beta)$. Therefore letting $\beta \searrow \beta_{cr}$ we obtain $1-\beta_{cr} = P_{e_1} (\tau_0<\infty)$.
 
   \subsubsection*{Critical and subcritical phases, $\beta \le \beta_{cr}$} 
  Recall that here $\lambda(\beta)=0$. Now let  $u$ be any solution to \eqref{eq:prob} with $u(0)=1$ and let  $\ol{u}(x) = (1-\delta_0(x))  u(x)$. 
   Then $$H_{-\infty} \ol{u}(x) = - \frac{u(0)}{2d}\ch_B(x).$$
    Letting $v(x)= \ol{u}(x) -\frac{1}{2d}R_0^{-\infty}(x,\ch_B)$, we observe that $v$ is harmonic for $H_{-\infty}$. Then by Lemma \ref{lem:martin},  is  $v = c A_0(x)$. In particular,
   $$ \ol{u} = c A_0(x) + \frac {R_0^{-\infty}(x,\ch_B)}{2d} =c A_0(x) + P_x (\tau_0<\infty),$$   
   where the second equality is due to Lemma \ref{lem:Rlambda}. On rewriting the equation $H_\beta u (0)=0$, we have 
  $$ \frac 1{2d} \sum_{|e|=1}\left (  \ol{u}(e)  - 1  \right) + \beta =0.$$
  This implies $u(e_1) = \ol{u}(e_1) = 1 - \beta$, as well as 
 $ c A_0(e_1) + P_{e_1} (\tau_0 <\infty) =1-\beta$. But $P_{e_1} (\tau_0<\infty) = 1-\beta_{cr}$. Therefore since $A_0(e_1)=1$, $c= \beta_{cr}-\beta$ and we have proved that 
 $$u(x) = (\beta_{cr}-\beta)A_0(x)+P_x(\tau_0<\infty).$$ 
 When $d=1,2$, the second term on the right-hand side is $1$ and $\beta_{cr}=0$, which leads to the formula  in the theorem. When $d\ge 3$ we have
  \begin{align*}
   u(x) &= 1- P_x (\tau_0=\infty) + (\beta_{cr}-\beta)\frac{  P_x (\tau_0 = \infty) }{P_{e_1} (\tau_0=\infty)}\\
   & =1 - P_x(\tau_0 =\infty) + \frac{\beta_{cr}-\beta}{\beta_{cr}}P_x (\tau_0=\infty)\\
    & = 1 -\frac{\beta} {\beta_{cr}}P_x (\tau_0=\infty).
    \end{align*}
               \end{proof} 
  \subsection{Proof of Theorem \ref{th:TR_theorem}}
  \begin{proof}
  The resolvent for  $(H_\beta - \lambda(\beta) )^{\psi_\beta}$ is equal to $\frac{1}{\psi_\beta(x)}R^\beta _{\lambda +\lambda(\beta)}(x,y)\psi_\beta (y)$. In particular, the expected time at $0$ starting from $0$ is equal to $R^\beta _{\lambda (\beta)}(0,0)= I^\beta(\lambda(\beta)) = \frac{ I(\lambda(\beta))}{1-\beta I(\lambda(\beta))}$. It is easy to see that $\beta_{cr} = \frac{1}{I(0^+)}$. Therefore, if $\beta< \beta_{cr}$, it follows that $I^\beta(\lambda (\beta)) = I^\beta (0)<\infty$, so that $X$ is transient. If $\beta= \beta_{cr}$, then $I^\beta (\lambda (\beta)) = I^{\beta_{cr}}(0)=\infty$, and if $\beta > \beta_{cr}$, then since $\psi_\beta \in \ell^2 (\Z^d)$ and $\psi_\beta^2$ is an invariant measure for $ (H_\beta -\lambda (\beta) )^{\psi_\beta}$, it follows that the process is positive recurrent. Conversely, any invariant density $\mu$ must satisfy $(H_\beta - \lambda (\beta)) \frac{\mu}{ \psi_\beta }=0$, that is $v=\frac{\mu}{\psi_\beta}$ is a positive solution to $(H_\beta - \lambda(\beta) ) v=0$.   Since by Theorem \ref{th:psib_desc} such a solution is equal to $c \psi_\beta$, it follows that the process is positive recurrent if and only if $\mu = c \psi_\beta^2$ for some $c$.  We know that this condition holds for $\beta> \beta_{cr}$. It remains to check $\beta =\beta_{cr}$. But by Theorem \ref{th:psib_desc}-(ii), $\psi_{\beta_cr}(x) = P_x (\tau_0 < \infty)$, and it is well know that this decays as $|x|^{2-d}$, hence square  integrable if and only if $d\ge 5$. 
 \end{proof} 
\subsection{Proof of Theorem \ref{th:Zbeta}}
\begin{proof}
Part \ref{Zbeta_cor} of the theorem follows directly from \eqref{eq:Zbeta} and Proposition \ref{pr:pbeta00}.

We turn to the proof of  part \ref{Zbeta_martin}.  Recall the function $A_\lambda$ defined in \eqref{eq:Alambda}. Since for every constant  $c>-\lambda$, we have  $\int_0^\infty e^{-(\lambda +c)t} dt = (\lambda+ c)^{-1}$, then from  \eqref{eq:RlambdaFourier} we obtain 
  $$
    A_\lambda(w) =  \int_0^\infty e^{-\lambda t}  \underset{=\tilde \psi_t(w)}{\underbrace{\frac{1}{\pi^{d}}\int_{[0,\pi]^d} e^{-\Phi(\varphi) t} \left ( 1- \cos \scp{\varphi ,w} \right) d \varphi }}dt.
   $$
    In other words, $A_\lambda(w)$ is the Laplace transform of $t\to \tilde \psi_t(w)$. 
  By monotone convergence, 
    \begin{equation} 
    \label{eq:A0} 
    A_0(x)=\int_0^\infty \tilde \psi_t(x) dt <\infty.
    \end{equation}
      We need some estimates on $\tilde\psi$. Below $c$ denotes a  positive whose value may change. Recall that 
$$\alpha^2 \ge 1-\cos \alpha = 2 \sin ^2 \frac \alpha 2 \ge c \alpha^2,$$
  the last  inequality holds for $\alpha \in [0,\pi]$.    Therefore $1- \cos \scp{\varphi ,w}\le \scp{\varphi,w}^2 \le |\varphi|^2 |w|^2$ and 
        $\Phi(\varphi)=\frac{2}{d}\sum_{j=1}^d \sin^2 (\frac {\varphi_j }{2}) \ge c |\varphi|^2$.
  Thus, 
             \begin{align}
             \nonumber
               \tilde \psi_t(w) &\le  c|w|^2  \int_{[0,\pi]^d} e^{-c |\varphi|^2 t} |\varphi |^2d \varphi \\
             \nonumber
              &\le c|w|^2   \int_0^{\infty} e^{-c r^2  t} r^{2} r^{d-1} dr\\
         \label{eq:tildepsi}
              &=c |w|^2 \int_0^{\infty} e^{-c  u  t}  u^{d/2} du= |w|^2  O(t^{-(d/2+1)}),~\mbox{ as } t\to\infty.
              \end{align}
              We recall that $R_\lambda^\beta(x,\ch)$ is the Laplace transform of the function $t\to Z_{\beta,t} (x)$. Furthermore, by \eqref{eq:resolventeqn} 
  $$R_\lambda^\beta(x,\ch) = \frac 1 \lambda \left (1+ \frac{\beta R_\lambda(x,0)}{1-\beta I(\lambda)}\right) = 
          \frac{ 1 - \beta I (\lambda) +\beta  R_\lambda(x,0)}{\lambda(1-\beta I(\lambda))}= \frac {1 -\beta A_\lambda (x)}{\lambda(1-\beta I(\lambda))}.$$
          Letting $x=0$ and recalling that $A_0(0)=0$, it follows that the Laplace transform of $t\to Z_{\beta,t}$ is $\frac{1}{\lambda(1-\beta I(\lambda))}$. Since $A_\lambda(x)$ is the Laplace transform of $t\to\tilde \psi_t (x)$, it follows that 
     \begin{align*}
       Z_{\beta,t}(x) & = Z_{\beta,t} -\beta (\tilde \psi_\cdot (x) * Z_{\beta,\cdot }) (t)\\
       &=  Z_{\beta,t}-\beta \left (  \int_0^ {(1-\epsilon) t}+ \int_{(1-\epsilon)t}^t \left ( Z_{\beta,t-s} \tilde \psi_s (x) \right ) ds\right) . \\
        &= Z_{\beta,t} - \beta \int_0^{(1-\epsilon)t}  \left ( Z_{\beta,t-s} \tilde \psi_s (x) \right ) ds + O(1)\int_{(1-\epsilon)t}^t \tilde \psi_s(x) ds \\
         &= Z_{\beta,t} - (1+o(1))\beta  Z_{\beta,t}\int_0^{(1-\epsilon)t} \tilde \psi_s(x) ds + o(Z_{\beta,t})\\
         & \sim  Z_{\beta,t}(1-\beta A_0(x)),~\mbox{ as }t\to\infty. 
      \end{align*}
       The second line  is due to the fact that $Z_{\beta,t} \le 1$, the third line is due to part \ref{Zbeta_cor} of the theorem and \eqref{eq:tildepsi},  and the   last line follows from \eqref{eq:A0}. 
       \end{proof}
 \subsection{Proof of Theorem \ref{th:small times}}
 \begin{proof} 
 We first prove convergence of finite dimensional distributions.  Let $0=t_0 < t_1<\dots < t_n \le T$ and let 
$0=x_0 ,x_1,\dots ,x_n \in \Z^d$. Then
 \begin{align}
 \nonumber 
   P_{\beta,t} (\cap_{i=1}^n \{X(t_i)=x_i\})& = \frac{\prod_{i=0}^{n-1}p_{\beta}(t_{i+1}-t_i,x_i,x_{i+1})}{Z_{\beta,t}}\times  Z_{\beta,t-t_n}(x_n)\\
   \nonumber
    &= \prod_{i=0}^{n-1}\frac{ p_{\beta}(t_{i+1}-t_i,x_i,x_{i+1}) Z_{\beta,t}(x_{i+1})}{ Z_{\beta,t} (x_i)} \times 
     \frac{Z_{\beta,t-t_n}(x_n)}{Z_{\beta,t}(x_n)} \\&
   \nonumber 
     \sim\prod_{i=0}^{n-1} \frac{1}{\psi_{\beta}(x_i)} p_{\beta}(t_{i+1}-t_i,x_i,x_{i+1}) \psi_\beta(x_{i+1}),\mbox{ as }t\to\infty,
     \end{align} 
where the last line is due to  theorem \ref{th:Zbeta}. 

The formula for $q_\beta$ follows from the fact that $q_\beta = \frac{1}{\psi_\beta(x)}p_{\beta}(t,x,y) \psi_\beta (y)$, and $p_\beta (t,x,y) = E_x \left [ \delta_y (X(t)) e^{\beta \I (t)}\right ] $.

Let $Q$ the distribution of $X$ under the transition function $q_\beta$.  Using the formula for  $q_\beta$ and the Markov property for $X$ under $E_0$, we conclude that 
$$ E^Q_0 \left[  \cap_{i=1}^n \{X(t_i)=x_i\} \right] = \frac{1}{\psi_\beta (0)}E_0 \left [ \psi_\beta (X(T))e^{\beta \I (T)} ; \cap_{i=1}^{n-1} \{X(t_i)=x_i\}\right].$$
Since $\psi_\beta (0)=1$, we have that $Q_0|_{{\cal F}_T}\ll P_0|_{{\cal F}_T}$ and that the Radon-Nikodym derivative is  $\psi_\beta (X(T)) e^{\beta \I (T)}$. 

To prove tightness, it is enough to show that for any $\epsilon>0$, there  exists some set $K_\epsilon\in {\cal F}_T$, compact in the topology on $D[0,T]$ such that 
 $$\lim_{\epsilon \to 0} \liminf_{t\to\infty} P_{\beta,t}(K_\epsilon)=1.$$ 
 Observe that $d P_{\beta,t}|_{{\cal F}_T} /d P_0|_{{\cal F}_T} = \frac{e^{\beta \I (T)}E_{X(T)} e^{\beta \I (t-T)}}{Z_{\beta,t}}$. 
It then follows from Theorem \ref{th:Zbeta}-(ii) and Fatou's lemma, that 
$$  \liminf_{t\to\infty} P_{\beta,t} (K_\epsilon) \ge E_0 \left [ \psi_\beta (X(T)) e^{\beta \I (T)} ; K_\epsilon\right] = Q_0 (K_\epsilon),$$
For each $\epsilon$, let $K_\epsilon$ be a compact set such that $P_0(K_\epsilon) \ge 1-\epsilon$. Then since $Q_0 |_{{\cal F}_T}\ll P_0|_{{\cal F}_T}$, it follows that $\lim_{\epsilon\to 0} Q_0( K_\epsilon)=1$, and the result follows. 
 \end{proof}
 
  \subsection{Proof of Theorem \ref{th:scaling_limit}} 
   The proof of the theorem consists of several stages. 
   \subsubsection{Tightness} 
  \label{sec:tightness} 
   \begin{lemma}
   \label{lem:tightness}
    Let  $\{Z^{(n)}_s:s \in [0,1]\}_{n\in\N}$ be a sequence of  continuous-time Markov chains with $Z^{(n)}$ having a countable state space ${\cal S}_n\subset \R$. Suppose that 
    \begin{enumerate} 
 \item $\{Z^{(n)}_0\}_{n\in \N}$ are tight. 
 \item There exists $t_0>0$, such that $\sup_{x\in {\cal S}_n} E |Z_t^{(n)} - Z_0^{(n)}|\le f(t)$ if $t\le t_0$, and  $f$ satisfies  $\lim_{t\searrow 0} f(t) =0$.  
 \end{enumerate} 
 Then $\{Z^{(n)}\}_{n\in\N}$ is  tight in $D[0,1]$. 
   \end{lemma} 
   \begin{proof}
     We will apply Aldous tightness criterion as it appears in \cite[Theorem 34.1]{bass}. We need to show the following: 
     \begin{enumerate} 
     \item For every $s \in [0,1]$, $\{Z^{(n)}_s\}_{n\in\N}$  is a tight sequence. 
     \item If  $\tau_n$ is a stopping time for $Z^{(n)}$, and  $\{\delta_n\}_{n\in\N}$ is a deterministic sequence decreasing to $0$, then 
      $ Z^{(n)}_{\tau_n+ \delta_n} - Z^{(n)}_{\tau_n}\to 0$ in probability as $n\to\infty$.
      \end{enumerate} 
      To prove the first, observe that $|Z^{(n)}_t| \le \sum_{j=1}^K |Z^{(n)}_{t_j} - Z^{(n)}_{t_{j-1}}| + |X^{(n)}_0|$, with $t_0=0<t_1<\dots< t_K =t$ and  $t_{j+1}-t_j < \delta$, where $\delta$ is chosen so that  $\sup_{t \le \delta }  f(t)=c < \infty$. Thus, 
      $\{|Z^{(n)}_t |> R\} \subset \{|Z^{(n)}_0| >R/(K+1)\} \cup_{j=0}^K \{ |Z^{(n)}_{t_{j+1}} - Z^{(n)}_{t_j}|  > R/(K+1)\}$. 
     So that by the Markov property, Chebychev, and the second condition, we have that 
     $$ P(|Z^{(n)}_t |> R) \le P( |Z^{(n)}_0| >R/(K+1)) + \frac{ K c}{R / (K+1)}.$$ 
     The right-hand side is independent of $n$, and tends to $0$ as $R\to\infty$. This completes the proof of the first. \\
     As for the second, it immediately follows from the second condition in the theorem, the  Markov property and the fact that $Z^{(n)}$ has a countable state space. 
         \end{proof}

We will now show that the processes $\{X^{(n)}\}$ in the statement of Theorem \ref{th:scaling_limit} are tight in $D[0,1]$. by showing that they satisfy the conditions of Lemma \ref{lem:tightness}. Specifically we will show 
\begin{prop}
\label{pr:translation_bd}
 Assume $d=1$ and $\beta <0$. Then  for any $x_0\in \Z^d$, 
  $$E^Q_{x_0}  |X_t - X_0|\le   \int_0^t P_0 (X_s = 0) ds -\beta \int_0^t Z_{\beta,s}  ds.$$ 
\end{prop} 
To continue we  recall the definition of the rescaled $D[0,1]$-valued  processes $X^{(n)}$ given by $X^{(n)}_t = X_{nt}/\sqrt{n}$ for $t \in[0,1]$, defined in the first paragraph of Section \ref{sec:scaling}. Observe that if $\nu_n$ is an initial distribution for $X$, then it induces an initial distribution $\nu_1^{(n)}$ for $X^{(n)}$, a Borel measure on $\R$,  given by the relation 
\begin{equation} 
\label{eq:nu1}  \nu^{(n)}_1 (A) = \sum_{\{z \in \Z: z/\sqrt{n}  \in A\} } \nu_n (z)= \nu_n ( \sqrt{n} A).
\end{equation} 
Note that the mapping $\nu_n \to \nu^{(n)}_1$ from the space of probability distributions on $\Z$ to the set of Borel  measures on $\R$ is one-to-one, and therefore $\nu^{(n)}_1$ uniquely determines $\nu_n$. It follows from this proposition and Lemma \ref{lem:tightness} that 
\begin{corollary}
\label{cor:XQ_tight}
Let $\{\nu_n:n\in\N\}$ be an sequence of initial distributions on $\Z$ such that $\{\nu^{(n)}_1:n\in\N\}$ is  tight. Then the family $\{Q_{\nu_n} (X^{(n)} \in \cdot) : n\in\N\}$ is tight. 
\end{corollary} 
To prove Proposition \ref{pr:translation_bd} we will need the following: 
\begin{lemma}
\label{lem:coupling}
For every $x \in\Z_+$, there a exists coupling $(X,X')$ such that under  $Q_x$, the distribution of the  process $X'=\{X_t':t \ge 0\}$ coincides with the distribution of $X$ under $Q_0$, and  
$|X_t| \ge |X'_t|,~\mbox{ for all }t$, $Q_x$-a.s. 
\end{lemma} 
\begin{proof}
 Let $X'$ be independent of $X$ until  (which may never happen) they either meet, or are mirror images of each other. If they meet first, then they coalesce, while if they are mirror images of each other, then the continue as such. In either case, $|X_t|\ge |X'_t|$ for all $t$. 
\end{proof} 
\begin{proof}[Proof of Proposition \ref{pr:translation_bd}]
  Let ${\cal L} = H_{\beta}^{\psi_{\beta}}$. If  $f:\Z \to \R$ and $x\in \Z$, we have 
   $${\cal L} f (x) = \frac{\psi_{\beta}(x+1)}{2\psi_{\beta}(x)} (f (x+1) - f(x) ) + \frac {\psi_{\beta}(x-1)} {2\psi_{\beta}(x)} (f(x-1) - f(x)).$$

   To obtain the bound on the increment, fix $x_0 \ge 0$, and let $f(x) = |x-x_0|$. 
   
   When $x>x_0$, we have $f(x+1) - f (x) =1$ and $f(x-1)-f(x)=-1$. When $x<x_0$, the signs are changed, and when $x=x_0$, $f(x+1) -f(x)=f(x-1)-f(x) = 1$. Therefore, we have
  \begin{align*} {\cal L} f (x) 
    & =  \delta_{x_0} (x)  \frac {\psi_{\beta}(x_0+1) + \psi_{\beta}(x_0-1)}{ 2\psi_{\beta}(x)}\\
     &  +\frac{1}{ 2\psi_{\beta}(x)} (\psi_{\beta}(x+1) - \psi_{\beta}(x-1))\left ( \ch_{\{x>x_0\}} -\ch_{\{x<x_0\}})\right),
    \end{align*} 
    and since $\psi_{\beta} (x) = 1+ |\beta| |x|$, it follows that for all $x$, 
    \begin{align} 
    \nonumber 
    &\psi_{\beta} (x_0+1) + \psi_{\beta} (x_0-1)= 2\psi_{\beta} (x) + 2\beta \delta_0(x_0),\mbox{ and }\\
    \label{eq:psi_diff}& \psi_{\beta} (x+1) - \psi_{\beta} (x-1)  = 2|\beta| \sgn(x),
    \end{align}
    where, as usual, $\sgn (x) = 1$ if $x>1$, $\sgn(x)=-1$ if $x<0$, and $\sgn(x)=0$ if $x=0$.  Therefore  
     \begin{equation} 
     \label{eq:Lf_bd} 
    | {\cal L} f (x)| \le (1+ |\beta|\delta_0 (x_0))\delta_{x_0} (x)+\frac{|\beta|}{\psi_{\beta} (x)} \ch_{\{x\ne x_0\}}.
     \end{equation} 
     Next, since ${\cal L}$ is the generator of $X$ under $Q$, it follows that for any bounded function $g$, with ${\cal L}g$ bounded, 
     $ g(X_t) - \int_0^t {\cal L} g (X_s) ds$ is a $Q$-martingale. In particular, $g_N = \min (f,N)$ is clearly such a function and if we let $$\tau_N=\inf\{t\ge 0: |X_t - x_0| > N-2\},$$ then 
      $$E^Q_{x_0} f (X_{t \wedge \tau_N}) = E^Q_{x_0}  g_N (X_{t \wedge \tau_N} ) =E^Q_{x_0}  \int_0^{t\wedge \tau_N}  {\cal L} g_N(X_s) ds= E^Q_{x_0} \int_0^{t \wedge \tau_N} {\cal L} f (X_s) ds.$$      
      Now $f (X_{t\wedge \tau_N}) \to f (X_{t-})$, and  it follows from Fatou's lemma that 
      $$ E^Q_{x_0} f (X_{t-}) \le \lim_{N\to\infty} E^Q_{x_0}  \int_0^{t \wedge \tau_N} {\cal L} f (X_s) ds.$$ 
      As from \eqref{eq:Lf_bd} is ${\cal L} f $ is uniformly bounded, it follows from right-continuity of the paths and dominated convergence that the right-hand side converges to $$E^Q_{x_0} \int_0^t {\cal L} f (X_s) ds,$$ so that 
       $$  E^Q_{x_0} f (X_{t-})  \le E^Q_{x_0} \int_0^t {\cal L} f (X_s) ds.$$ 
       Replacing $t$ by $t+\epsilon$ and letting $\epsilon\ge 0$, it follows again from Fatou's lemma, that 
       $$ E^Q_{x_0} f (X_t) \le E^Q_{x_0} \int_0^t {\cal L} f (X_s)ds.$$ 
     From this, the definition of $f$ and \eqref{eq:Lf_bd} we obtain the upper bound: 
     \begin{align} \nonumber E_{x_0}^Q |X_t -x_0| &\le \int_0^t (1+\delta_0 (x_0) |\beta|)Q_{x_0} (X_s = x_0)+|\beta|  E^Q_{x_0} [\frac{1}{\psi_{\beta}(X_s)};X_s \ne x_0]ds\\
     \nonumber
      & = \int_0^t (1+\delta_0(x_0) |\beta| ) p_{\beta}(s,x_0,x_0) \\
   \nonumber    &\quad\quad+\frac{ |\beta|}{\psi_{\beta}(x_0)}E_{x_0} [ e^{\beta \I (s) }(1-\delta_{x_0}) (X_s) ] ds\\
      \nonumber 
      & = \int_0^t(1+\delta_0(x_0) |\beta|  - \frac{|\beta| }{ \psi_{\beta}(x_0)}) p_{\beta}(s,x_0,x_0) + \frac {|\beta|}{\psi_{\beta}(x_0)} Z_{\beta,s} (x_0)ds\\
      \label{eq:distance}
      &\le  \int_0^t P_0 (W_s = 0) ds + \frac{|\beta|} {\psi_{\beta}(x_0)}\int_0^t Z_{\beta,s} (x_0) ds.
      \end{align}   
      Observe that $Z_{\beta,t} (x_0) = E_{x_0} [ e^{\beta \I(t)} ]=\psi_{\beta} (x_0) E^Q_{x_0} \frac{1}{\psi_{\beta}(X_t)}$. From Lemma \ref{lem:coupling} it follows that $E^Q_{x_0} \frac{1}{\psi_{\beta}(X_t)} \le E^Q_0   \frac{1}{\psi_{\beta}(X_t)}= Z_{\beta,t}$. Thus, 
      $ Z_{\beta,t} (x) \le \psi_{\beta}(x) Z_{\beta,t}$, and so the result follows from this and the right-hand side of \eqref{eq:distance} 
       \end{proof} 
       Let $Y^{(n)}$ be the process defined through  $Y^{(n)}_t = (X^{(n)}_t)^2$. 
     In order to simplify some of the arguments and in particular avoid convergence to a diffusion with singular drift, we will work with $Y^{(n)}$  rather than $X^{(n)}$ itself. We first need to show that the tightness of $X^{(n)}$ is preserved under the square map.
  \begin{prop}\label{pr:continuous}
The map $x\mapsto x^2$ from $D[0,1]$ into $D[0,1]$ where $x^2$ is defined by $x^2(t)=(x(t))^2$ is continuous in the Skorokhod topology.
\end{prop}

\begin{proof}
Recall that the Skorokhod topology can be generated by the metric $$\rho(x,y)=\inf_{\lambda\in\Lambda}\{\|\lambda-I\|\vee \|x\circ\lambda-y\|\}$$ where $\Lambda$ is the set of all continuous, strictly increasing functions from $[0,1]$ onto $[0,1]$, $I$ is the identity function on $[0,1]$, and $\|\cdot\|$ is the usual sup norm. If $x,y\in D[0,1]$, a straightforward calculation shows that $\|x^2-y^2\|\leq \|x-y\|(\|x\|+\|y\|)$. Now $ x^2 \circ \lambda= (x \circ \lambda)^2$, therefore  $\|(x^2 \circ\lambda) -y^2\| \leq \|x\circ\lambda-y\|(\|x\|+\|y\|)$.  This gives: 
$$\rho(x^2,y^2)\leq \inf_{\lambda\in\Lambda}\{\|\lambda-I\|\vee \|x\circ\lambda-y\|(\|x\|+\|y\|)\}\leq (1+\|x\|+\|y\|)\rho(x,y),$$
and the result follows.
\end{proof}
From this we obtain the analog of Corollary \ref{cor:XQ_tight} for $Y^{(n)}$. Before stating the result, observe that if $\nu_n$ is an initial distribution for $X$, then in analogy to \eqref{eq:nu1}, it  induces an initial distribution $\nu^{(n)}_2$  for $Y^{(n)}$,  a Borel measure on $\R_+$, given by the following relation: 
\begin{equation} 
\label{eq:nu2} 
\nu^{(n)}_2(A) = \nu^{(1)}_n ( \{ x \in \R : x^2 \in A\}) = \sum_{\{z \in \Z: z^2/ n \in A\}} \nu_n (z).
\end{equation} 
\begin{corollary}
\label{eq:tightness_squared} 
Let $\{\nu_n:n\in\N\}$ be an sequence of initial distributions on $\Z$ such that $\{\nu^{(n)}_2:n\in\N\}$ is  tight. Then the family $\{Q_{\nu_n} (Y^{(n)} \in \cdot) : n\in\N\}$ is tight.  \end{corollary} 
\begin{proof}
 Let $F$ denote the mapping $x\to x^2$ from Proposition \ref{pr:continuous}. Fix $\epsilon>0$ such that $Q(X^{(t)} \in K )>1-\epsilon$ for all $t$. But since $F$ is continuous, $F(K)$ is continuous, and it follows that $ Q ( F (X^{(t)}) \in F (K)) \ge Q(X^{(t)} \in K )>1-\epsilon$.
\end{proof}
\subsubsection{Martingale Problem} 
\label{sec:martigale_prob}
Define the second order differential operator ${\cal L}^{(\infty)}$ through 
$${\cal L}^{(\infty)} g (y) =2y g''(y) + 3 g'(y).$$ 
Also, for $t \ge 0$ and $n\in\N$, let 
$${\cal G}^{(n)}_t=\sigma(X^{(n)}_s: s\le t) = \sigma (X_{s}:s\le nt).$$
We have
\begin{lemma}
 Let $g \in C_c^\infty(\R)$, and $t>s>0$ and $A \in {\cal G}^{(n)}_s$. Then 
  $$E^Q_{\nu_n} \left [ g (Y^{(n)}_t) - \int_0^t {\cal L}^{(\infty)} g (Y^{(n)}_s) ds; A \right]  = o(n^{-1/2}).$$
\end{lemma} 
\begin{proof} 
Given any initial distribution $\nu$ for $X$, 
 $ f(X_t) - \int_0^t {\cal L} f (X_s) ds$ is a $Q_\nu$-martingale.  We will convert this statement into the family of processes $\{X^{(n)}\}_{n\in\N} $.  We can write  $f (X^{(n)}_t) = h (X_{nt})$, where $h(z) =   f (z / \sqrt{n}) $. Then 
 $$M^{(n)}_t =  f (X^{(n)}_t) - \int_0^{nt} {\cal L} h (X_s) ds$$  is a ${\cal G}^{(n)}_t$-martingale under $Q_{\nu}$.  
  Now ${\cal L} h  (z) = \frac{1}{\psi_\beta (x)} \left ( E (\psi_{\beta}h) (x+Z) -  (\psi_{\beta}h)(x)\right) +\beta \delta_0 h(x) $, where $Z$ is $1$ with probability $\frac 12$ and $-1$ with probability $\frac 12$.
  From the Taylor expansion of $f$, we conclude that 
  \begin{align*} 
  &(\psi_{\beta}h) (x+Z) = \psi_{\beta}(x+Z) \\
  &\quad\times  \left ( f (x/\sqrt{n}) + f (x/\sqrt{n}) \frac{Z}{\sqrt{n}} + f'' (x/\sqrt{n}) \frac{1}{2n} +\frac 16 f^{(3)}(x'/\sqrt{n})\frac{Z}{n^{3/2}}\right).
  \end{align*}
   where $x' \in (x-1/\sqrt{n},x+\sqrt{n})$. Since $E \psi_{\beta} (x+Z) - \psi_{\beta}(x)+\beta \delta_0 =0$, it follows that 
   $${\cal L} h (x) = E \frac{\psi_{\beta} (x+Z)}{\psi_\beta(x)}\left ( f'(x/\sqrt{n}) \frac{Z}{n} + f''(x/\sqrt{n}) \frac{1}{2n}\right) + D^{(n)}_1(x),$$
   where (slightly abusing notation, note that here $x'$ is a random variable which is a function of $Z$ and $n$):
    $$D^{(n)}_1(x)=n^{-3/2}  E[  \psi_{\beta}(x+Z)Z  \frac 16 f^{(3)}(x'/\sqrt{n})],$$
    and in particular,   
    \begin{equation} 
    \label{eq:D1} 
    |D^{(n)}_1(x) |  \le cn^{-3/2}  \|f^{(3)}\|_\infty,
    \end{equation}  where  $c$ is a universal constant. 
    From   \eqref{eq:psi_diff}, we have that $E[ \psi_{\beta} (x+Z) Z] = |\beta| \sgn(x) $, so that
 $$ {\cal L} h (x) = \frac 1{2n}  f'' (x/\sqrt{n}) + f'(x/\sqrt{n}) \frac{|\beta| \sgn (x)}{\sqrt{n} \psi_{\beta}(x)} +D^{(n)}_1(x)+D^{(n)}_2(x), $$
 where 
 $$ D^{(n)}_2(x) = \frac 1{2n}  f'' (0)  |\beta| \delta_0 (x).$$ 
 Therefore 
 \begin{align*} 
 M^{(n)}_t & =  f (X^{(n)}_t) - \int_0^{nt}  \frac{1}{2n
 } f'' (X_s /\sqrt{n} ) +  f'(X_s/\sqrt{n}) \frac{|\beta| \sgn (X_s)}{\sqrt{n} \psi_{\beta}(X_s)} ds - E^{(n)}_1(t)\\
  & = f (X^{(n)}_t) - \int_0^{t}  \frac{1}{2} f'' (X^{(n)}_s)  + f'(X^{(n)}_s) \frac{\sqrt{n}|\beta| \sgn (X^{(n)}_s)}{ \psi_{\beta}(X_{ns})} ds- E^{(n)}_1(t),
   \end{align*} 
where 
   $$ E^{(n)}_1 (t) = \int_0^{nt} D^{(n)}_1(X_s) ds+ \int_0^{nt} D^{(n)}_2(X_s) ds.$$ 
  Observe that 
  $$ E^Q_{\nu} |E^{(n)}_1 (t)| \le  c n^{-1/2} \|f^{(3)}\|_\infty + |\beta| \frac{|f''(0)|}{n}  \int_0^{nt} Q_{\nu} (X_s =0)ds.$$ 
From Lemma \ref{lem:coupling},  $Q_\nu (X_s=0) \le Q_0(X_s=0)$ and since $X$ is transient under $Q_0$, the integral is finite, therefore 
  $$ \sup_{\nu}  E^Q_{\nu} |E^{(n)}_1 (t)|  \le c(f,\beta) n^{-1/2}.$$ 
  In order to simply the drift expression, and prove the lemma,  we will take  $f= g(x^2)$. Observe then that $f'(x) = 2x g'(x^2)$ and that $\frac 12 f''(x) =2x^2   g''(x^2)+  g'(x^2)$. Letting $Y^{(n)}=(X^{(n)})^2$, we obtain 
    \begin{align*}  M^{(n)}_t & = g (Y^{(n)}_t) - \int_0^t 2 Y^{(n)}_s g''(Y^{(n)}_s) +  g'(Y^{(n)}_s) +  g'(Y^{(n)}_s) \frac{2|\beta X_{ns}|}{ \psi_{\beta}(X_{ns})} ds- E^{(n)}_1(t).\\
     & = g (Y^{(n)}_t) - \int_0^t 2 Y^{(n)}_s g''(Y^{(n)}_s) +   3 g'(Y^{(n)}_s) ds - E^{(n)}_1(t)-E^{(n)}_2(t),
    \end{align*} 
   where 
   $$ E^{(n)}_2(t) = 2 \int_0^t g'(Y^{(n)}_s) \left ( \frac{|\beta X_{ns}|}{ \psi_{\beta}(X_{ns}) }-1 \right) ds .$$ 
   Observe that 
   $$ |E^{(n)}_2(t)|\le  \|g'\|_\infty\int_0^{t}  \frac{1}{\psi_{\beta}(X_{ns})^2}ds.$$  
   Now 
   $$ E^Q_{\nu} |E^{(n)}_2(t)| \le \|g'\|_{\infty} \int_0^{t} E^Q_{\nu} \frac{1}{\psi_{\beta}(X_{ns})^2} ds.$$
    Lemma \ref{lem:coupling} gives  $ E^Q_{\nu} \frac{1}{\psi_{\beta}(X_{ns})^2} \le E^Q_0 \frac{1}{\psi_{\beta}(X_{ns})^2}\le E^Q_0 \frac{1}{\psi_{\beta}(X_{ns})} = Z_{\beta,ns}$, and therefore 
    $ E^Q_{\nu} |E^{(n)}_2(t) | \le \|g'\|_{\infty} \int_0^t  Z_{\beta,ns} ds$. For $u>1$, $Z_{\beta,u} \le \frac{c}{\sqrt{u}}$, therefore 
    $$ \int_0^t Z_{\beta,ns} ds \le \frac{1}{\sqrt{n}} +\int_{1/\sqrt{n}}^1 \frac{c}{\sqrt{ns}}ds \le \frac{c}{\sqrt{n}}, $$
    and the constant $c$ depends only on $\beta$.    It follows that $\sup_{\nu}E^Q_{\nu} |E^{(n)}_2(t)|\le\|g'\|_\infty \frac{c}{\sqrt{n}}$. Writing $E^{(n)}_t = E^{(n)}_1(t) + E^{(n)}_2(t)$, then $\sup_{\nu} E^{Q}_\nu |E^{(n)}_t|= O(n^{-1/2})$, and since $M^{(n)}_t$ is a ${\cal G}^{(n)}_t$ martingale under $Q_{\nu_n}$, we have that 
        \begin{align*} 
    0&=  E^Q_{\nu_n} [ M^{(n)}_t; A] \\
     &  = E^Q_{\nu_n} [ g (Y^{(n)}_t) - \int_0^t {\cal L}^{(\infty)} g (Y_s) ds; A]\\
     & ~~~~-E^Q_{\nu_n} [ E^{(n)}_t \ch_A ],
     \end{align*}
    and the result follows. 
    \end{proof}
    \subsubsection{Convergence of Markov Chains} 

We fix some notation. Let $R^2$ denote the process $\{R_t^2:t\ge 0\}$, where $R$  the Bessel-$3$ process introduced in Section \ref{sec:scaling}. We also write  $(Z,\nu)$ for a Markov chain the Markov process $Z$ with initial distribution $\nu$. 

From \cite[Theorem 8.2.3, p. 372]{EK} We first observe that the martingale problem for ${\cal L}^{(\infty)}$ is well-posed. It is also well-known that ${\cal L}^{(\infty)}$ is the generator of $R^2$.  The following is an immediate consequence of  \cite[Theorem 4.8.10 , p 234]{EK}.
\begin{prop}
\label{pr:Y2_conv}
Suppose that  $\{\nu_n:n\in\N\} $   is such that $ \nu^{(n)}_2\Rightarrow \nu^{(\infty)}_2$ for some Borel probability measure $\nu^{(\infty)}_2$ on $\R_+$.  Then 
$$Q_{\nu_n} ( Y^{(n)} \in \cdot ) \Rightarrow R^2,$$ 
and the process on the righthand side has initial distribution $\nu^{(\infty)}_2$. 
\end{prop} 
With this, we are ready to prove Theorem \ref{th:scaling_limit}. 
\begin{proof}[Proof of Theorem \ref{th:scaling_limit}]
The tightness of $\{X^{(n)}\}_{n\in\N}$ under $Q_0$ follows from Corollary \ref{cor:XQ_tight}. We need to identify the finite dimensional distributions. For this purpose, we will use the convergence of $\{Y^{(n)}\}_{n\in\N}$ to $R^2$. Suppose that $f \in C_b (\R)$. Let $f_+ = f ( |x|)$ and $f_-(x) = f (-|x|)$. Thus, $f_\pm$ are symmetric functions, coinciding with $f$ on positive half-line and negative half-line, respectively. Clearly, 
$$ E^Q_0[ f (X^{(n)}_t)] = E^Q_0 [f (X^{(n)}_t) \ch_{\{X_{nt} > 0\}}]+  E^Q_0 [f (X^{(n)}_t) \ch_{\{X_{nt} < 0\}}]+ f (0)Q_0 (X_{nt}=0).$$ 
The third term on the righthand side tends to $0$, as $X$ is not positive recurrent under $Q_0$. Due to the symmetry of $X$ under $Q_0$, we can write 
 \begin{align*} 
  E^Q_0 [f (X^{(n)}_t) \ch_{\{X_{nt} > 0\}}] &= \frac 12 \left ( E^Q_0 [f_+(X^{(n)}_t) ]-Q_0 (X_{nt}=0)\right) \\
   & =\frac 12 E^Q_0 [f_+(X^{(n)}_t) ] + o(1).
   \end{align*}
 Therefore, 
 $$  E^Q_0 [f (X^{(n)}_t)]  =  \frac 12 E^Q_0 [f_+(X^{(n)}_t) ] + \frac 12  E^Q_0 [f_-(X^{(n)}_t) ] + o(1),$$
However, since $f_\pm$ are symmetric, we can rewrite this in terms of $Y^{(n)}$. That is, 
$$ E^Q_0 [f (X^{(n)}_t)] = \frac 12 E^Q_0 [\frac 12 \left ( f_++f_-)(\sqrt{Y^{(n)}_t}) (\sqrt{Y^{(n)}_t}) \right) \ch_A (X_{nt}) ]+ o(1).$$ 
By choosing $A=\R$, we conclude that 
\begin{equation}
\label{eq:XQ_onemarg} 
 Q_0 (X^{(n)}_t \in \cdot )\Rightarrow JR_t.
\end{equation} 
In order to complete the proof, we need to show that the finite dimensional distributions of $X^{(n)}$ converge under $Q_0$. To this end, similarly to the definition of $\nu_n,~\nu^{(n)}_1$ and $\nu^{(n)}_2$, define the following signed Borel measures on $\R$: 
$$\rho_n (A)= E_0^Q[ \frac 12 (f_++f_-) (\sqrt{Y^{(n)}_t}) \ch_A (X_{nt})],$$
$\rho^{(n)}_1(A) = \sum_{\{z \in A: z/\sqrt{n} \in A}\rho_n (z)\}$, and $\rho^{(n)}_2 (A) = \sum_{\{z: z^2/n \in A}\} \rho_n (z)$. Observe then that \eqref{eq:XQ_onemarg} could be rewritten as 
\begin{equation} 
\label{eq:rhon_lim} 
 \rho^{(n)}_1 \Rightarrow E[ f(JR_t)\ch_{\cdot} (JR_t) ].
\end{equation} 
Fix $0<t_1< \dots <t_l\le 1$, and let $f_1=f,f_2,\dots f_l$ be in $C_b(\R)$. We have 
$$ E^Q_0[ \prod_{j=1}^l f_j (X^{(n)}_{t_j})]=\int E_z^Q [\prod_{j=2}^l f_j (X^{(n)}_{t_j-t_1}) ]d \rho_n(z).$$ 
We will decompose the integral into three domains as follows. Since $\rho^{(n)}_1$ converges to an absolutely continuous signed measure on $\R$, for any $\epsilon>0$, there exists $\delta>0$ such that $\rho_n ( D_n) \le \epsilon$ for all $n$ large, where $D_n = \{z: |z|\le \delta \sqrt{n}\}$. Let $D_{n,+}= \{z\in \Z: z > \sqrt{n} \delta\}$ and $D_{n,-} = - D_{n,+}$. Then 
$$ E^Q_0 \prod_{j=2}^l f_j (X^{(n)}_{t_j})= \int_{D_n \cup D_{n,+}\cup D_{n,-}} E_z^Q [ \cdots] d \rho_n(z).$$ 
Since the functions $f_1,\dots,f_l$ are uniformly bounded, it follows that 
$$ | \int_{D_n} E_z^Q [ \cdots ] d \rho_n | \le c \epsilon.$$ 
Next, if $z \in D_{n,+}$, then we can write
 $$ E_z^Q [ \prod_{j=2}^l f_j (X^{(n)}_{t_j}) ] = E_z^Q [ \prod_{j=2}^l f_j (X^{(n)}_{t_j}) ; \tau_0 > n] +  E_z^Q [ \prod_{j=2}^l f_j (X^{(n)}_{t_j}) ; \tau_0 \le n] .$$ 
 But, 
$$ |E_z^Q [ \prod_{j=1}^l f_j (X^{(n)}_{t_j}) ; \tau_0 \le n] |\le c Q_z(\tau_0 \le n) \le c Q_{\lfloor \sqrt{n} \delta \rfloor}(\tau_0 < \infty),$$
and where the constant $c$ depends only on $f_2,\dots,f_l$. By transience of $X$ under $Q$, it immediately follows that the probability on the righthand side tends to $0$ as $n\to\infty$. In addition, 
\begin{align*} E_z^Q [ \prod_{j=2}^l f_j (X^{(n)}_{t_j}) ; \tau_0 > n]  \\
&  =E_z^Q [ \prod_{j=2}^l f_j (\sqrt{ Y^{(n)}_{t_j}}); \tau_0 > n] \\
& = E_z^Q [ \prod_{j=2}^l f_j (\sqrt{ Y^{(n)}_{t_j}})]+ o(1),
\end{align*}
 where $|o(1)|\le c Q_{\lfloor \sqrt{n} \delta \rfloor}(\tau_0 < \infty)$. 
Combining the two results, we conclude that 
$$ \int_{D_{n,+}} E_z^Q [ \prod_{j=2}^l f_j (X^{(n)}_{t_j-t_1})] d \rho_n(z) =\int E_z^Q [ \prod_{j=2}^l f_j (\sqrt{ Y^{(n)}_{t_j}})] d \rho_{n,+}(z) +o(1),$$
where $\rho_{n,+}(A) = \rho_{n} (A \cap D_{n,+})$.
A similar argument shows that 
$$ \int_{D_{n,-}} E_z^Q [ \prod_{j=2}^l f_j (X^{(n)}_{t_j-t_1})] d \rho_n(z)  =\int E_z^Q [ \prod_{j=2}^l f_j (\sqrt{-Y^{(n)}_{t_j}})] d \rho_{n,-}(z) + o(1),$$
with $\rho_{n,+}(A) = \rho_{n}(A \cap D_{n,-})$.  Altogether,
\begin{align*}  \int E_z^Q [ \prod_{j=2}^l f_j (X^{(n)}_{t_j-t_1})] d \rho_n(z) &= \int E_{\rho_{n,+}} [ \prod_{j=2}^l f_j (\sgn (X^{(n)}_0) Y^{(n)}_{t_j-t_1}] + \\
& \quad +E_{\rho_{n,-}} [ \prod_{j=2}^l f_j (\sgn (X^{(n)}_0) Y^{(n)}_{t_j-t_1}] +  o(1) +\epsilon O(1).
\end{align*} 
However, since $\rho^{(n)}_{1,\pm}$  converge to $A \to \frac 12 E_0 [f(_++f_-)(\pm R_t) \ch_A(\pm R_t)]$, respectively, and hence $\rho^{(n)}_{2,\pm}$ converge as well, it follows from Proposition \ref{pr:Y2_conv}, and \eqref{eq:rhon_lim} that 
$$E_{\rho_{n,+}+\rho_{n,-}} [ \prod_{j=2}^l f_j (\sgn (X^{(n)}_0) Y^{(n)}_{t_j-t_1}]\to E[ f (JR_t)E_{R_t} \prod_{j=2}^l f_j (JR_{t_j-t_1});R_t \ge \delta].$$ 
Thus, 
$$ E^Q_0[ \prod_{j=1}^l f_j (X^{(n)}_{t_j})] =  E[ f (JR_t) \prod_{j=2}^l f_j (JR_{t_j-t_1});R_t > \delta]+o(1)+\epsilon O(1).$$ 
The result now follows by letting $\epsilon$ (and consequently $\delta$) to $0$, and from the absolute continuity of $JR_t$. 
\end{proof}
 \subsubsection{Convergence of the Polymer} 
 \begin{proof}[Proof of Theorem \ref{th:scaling_limit}-(ii)] 
 In light of \eqref{eq:imhof}, we need to show that for $F:D[0,1]\to \R$  continuous and bounded, 
\begin{equation} 
\label{eq:required_limit} 
\lim_{n\to\infty}  \frac{E^Q_0[  F (X^{(n)}) /\psi_{\beta} (X^{(n)}_1) ]}{Z_{\beta,n}}= \frac{ E[F(JR) / R_1] }{E [1/R_1]}.
\end{equation} 
To prove this limit, let's first assume that $F$ is also nonnegative. Then 
  $$ \sqrt{n} E^Q_0 [F(X^{(n)}) / \psi_{\beta} (X^{(n)}_1)] \ge E [ \frac{ \sqrt{n}}{\psi_{\beta}(X^{(n)}_1)} F (X^{(n)}); |X^{(n)}_1|>\epsilon].$$

Now  $\psi_{\beta}(X^{(n)}_1) = 1 - \beta |X^{(n)}_1|$, and so 
$$\liminf_{n\to\infty}  \frac{\sqrt{n} \ch_{(\epsilon ,\infty)} (|X^{(n)}_1|)}{\psi_{\beta}(X^{(n)}_1)}  = \liminf_{n\to\infty} \frac{\sqrt{n}\ch_{(\epsilon ,\infty)} (X^{(n)}_1)}{1-\beta \sqrt{n} |X^{(n)}_1|} \ge \frac{ \ch_{(2\epsilon,\infty)}(X^{(n)}_1)}{-\beta |X^{(n)}_1|}.$$ 
Now  
$$ \frac{ E^Q_0 [ F (X^{(n)}) /\psi_{\beta} (X^{(n)}_1) ]}{Z_{\beta,n}}   \ge   \frac{1}{\sqrt{n}Z_{\beta,n}}  E^Q_0 [ \frac{\sqrt{n}}{\psi_{\beta}(X^{(n)}_1)}  F(X^{(n)}) ;|X^{(n)}_1|>\epsilon ] .$$ 
It follows from Fatou's lemma that 
$$ \liminf_{n\to\infty} \frac{ E^Q_0 [ F (X^{(n)}) /\psi_{\beta} (X^{(n)}_1) ]}{Z_{\beta,n}}  \ge \sqrt{\frac{\pi}{2}} E^Q_0[ \frac{F (JR)}{R_1}; R_1>2\epsilon ] .$$
Since $ E\frac{1}{R_1} =  \sqrt{\frac{2}{\pi}}$, it  follows that 
$$ \liminf_{n\to\infty} \frac{E^Q_0[  F (X^{(n)}) /\psi_{\beta} (X^{(n)}_1) ]}{Z_{\beta,n}} \ge \frac{E[  F (R)/R_1;R_1>2\epsilon] }{E[1/R_1]}.$$ 
Since $\epsilon$ is arbitrary and $Q(R_1=0)=0$, it follows from the monotone convergence theorem that 
\begin{equation}
\label{eq:fatou}  \liminf_{n\to\infty}  \frac{E^Q_0[  F (X^{(n)}) /\psi_{\beta} (X^{(n)}_1) ]}{Z_{\beta,n}} \ge \frac{E^Q_0[ F (JR)/R_1] }{E [1/R_1]}.
\end{equation} 
Now assume that  $F$ is bounded and continuous. Let  $c=\inf F$, then $F-c$ is nonnegative, so it follows from \eqref{eq:fatou} that 
$$ \liminf_{n\to\infty}  \frac{E^Q_0[  (F-c) (X^{(n)}) /\psi_{\beta} (X^{(n)}_1) ]}{Z_{\beta,n}}\ge  \frac{E [ (F (JR)-c) /R_1] }{E [1/R_1]}.$$
However, the lefthand side is equal to $ \liminf_{n\to\infty} \frac{ E^Q_0 F(X^{(n)}) / \psi_{\beta}(X^{(n)}_1) }{Z_{\beta,n}}-c$, and the righthand side is equal to $\frac{E^Q_0[ (c-F (R)) /R_1] }{E^Q_1 [1/R_1]}-c$. Thus, 
\begin{equation}
\label{eq:lower_bd} \liminf_{n\to\infty}  \frac{E^Q_0[  F (X^{(n)}) /\psi_{\beta} (X^{(n)}_1) ]}{Z_{\beta,n}}\ge  \frac{E [ F (JR) /R_1] }{E [1/R_1]}.
\end{equation}
Similarly, letting $c =\sup F$, then  $c -F$ is nonnegative, and so from \eqref{eq:fatou} we obtain 
$$ \liminf_{n\to\infty}  \frac{E^Q_0[  (c-F) (X^{(n)}) /\psi_{\beta} (X^{(n)}_1) ]}{R_{\beta,n}}\ge  \frac{E[ (c-F (JR)) /R_1] }{E [1/R_1]}.$$
The lefthand side is equal to $ c- \limsup_{n\to\infty} \frac{ E^Q_0 [F(X^{(n)}) / \psi_{\beta}(X^{(n)}_1)] }{Z_{\beta,n}}$, and the righthand side is equal to $c-\frac{E [ F (JR) /R_1] }{E^Q_1 [1/R_1]}.$
  Thus, 
  \begin{equation}\label{eq:upper_bd}
    \limsup_{n\to\infty} \frac{ E^Q_0[ F(X^{(n)}) / \psi_{\beta}(X^{(n)}_1)] }{Z_{\beta,n}} \le \frac{E[ F (JR) /R_1] }{E [1/R_1]}.
  \end{equation} 
  The desired limit \eqref{eq:required_limit} follows from the inequalities \eqref{eq:lower_bd} and \eqref{eq:upper_bd}.
 \end{proof} 
\section{Appendix}
 \begin{proof}[Proof of Lemma \ref{lem:Ilambda}]
  \begin{enumerate}
   \item  $d=1$. We first assume $ \lambda>0$ and then extend by analytic continuation. 
     When $d=1$, we can change variables by letting $\cos \varphi = \frac{z+z^{-1}}{2}$ where $z=e^{i \varphi}$ and $\varphi \in [0,2\pi)$. Then $dz = i z d \theta$. Thus,  
     \begin{align*} 
      I(\lambda) &= \frac{1}{2\pi} \int_0^{2\pi} \frac{d \varphi }{\lambda + 1- \cos (\varphi)}\\
       & = \frac{1}{\pi} \int_{C_1}  \frac{1/ (iz)  d z }{2\lambda +2- z - 1/z }\\
       &= \frac{1}{\pi i} \int_{C_1} \frac{dz} {z(2\lambda+2) - z^2 -1}.
       \end{align*} 
      $$I(\lambda)=\frac{1}{2\pi} \int_{C_1}\frac{dz}{(2\lambda+2)z-z^2-1}.$$
       Let $z_1,z_2$ denote the solutions to   $f(z)=(2\lambda+2)z-z^2-1=0$. Then $z_1z_2 =1$ and $z_1+z_2 = 2\lambda+2$, which implies that exactly one solution is inside the unit circle, denote it by $z_1$. Then we have 
       $$ I(\lambda) = -\frac{1}{\pi i} \int_{C_1} \frac{1}{(z-z_1)(z-z_2)}  = - 2 \mbox{Res}(f;z_1)=  \frac{2}{z_2-z_1}.$$ 
       Finally, $\frac{z_2- z_1}{2} =  \frac{ \sqrt{ (2\lambda+2)^2 -4}}{2}= \sqrt{(\lambda +2) \lambda}$. 
       $$ I(\lambda) = \frac{1}{ \sqrt{\lambda (2+\lambda)}}.$$
       
       \item  $d=2$.  In what follows $c$ denotes a positive constant whose value may change. 
       Observe that $\Phi(\varphi) = \sin^2 (\varphi_1/2) + \sin^2 (\varphi_2/2)$. Therefore, 
       $$ I(\lambda) = 4\pi^{-2} \int_{[0,\pi/2]^2} \frac{1}{\lambda+ \sin^2 (\alpha_1) + \sin^2 (\alpha_2)} .$$
       Change variables to $x_1 = \sin (\alpha_1)$ and $x_2=  \sin (\alpha_2)$, to obtain $d x_j/ d\alpha_j = \cos (\alpha_j)=\sqrt{1-x_j^2}$. Therefore 
       \begin{align*}  I(\lambda) &=4\pi^{-2} \int_{[0,1]^2} \frac{1}{\lambda + x_1^2 + x_2^2} \frac{ 1}{\sqrt{1-x_1^2}}\frac{ 1}{\sqrt{1-x_2^2}}dx_1 dx_2\\
         & = 8 \pi^{-2} \int_0^1 \int_{0^x_1}  \frac{1}{\lambda + x_1^2 + x_2^2} \frac{ 1}{\sqrt{1-x_1^2}}\frac{ 1}{\sqrt{1-x_2^2}}.
         \end{align*} 
       Change to polar coordinates to obtain 
       \begin{equation} 
       \label{eq:2dI_polar} 
        I(\lambda) = 8 \pi^{-2} \int_0^{\pi/4} \int_0^{1/\cos \theta} \frac{r}{\lambda + r^2}(1+ h(r,\theta)) dr d\theta,
        \end{equation}
 where $h(r,\theta) =  \frac{1}{\sqrt{1-(r \cos \theta)^2}}\frac{1}{\sqrt{1-(r\sin \theta)^2}}-1\ge 0$.  We can break the integral into two. Since we're interested in the behavior  when 
 $\lambda$ is near the origin, let us fix some $\delta>0$, and assume that  $|\lambda|< \delta^2/2$. We then write the integration domain as the union of $A= \{r \le \delta, \theta \in [0,\pi/4]\}$ and its relative complement $B$, and write $I_{A}$ and $I_{B}$ for the integrals over the respective sub domains. On $B$, function $\frac{r}{\lambda+r^2}$ is uniformly bounded, Hence 

 $$  |I_B| \le c(\int_0^1 \frac{1}{\sqrt{1-x^2}}dx)^2\le c.$$
 We turn to integration on $A$. We integrate by parts: 
 \begin{align*} 
 I_{A} & =\frac{8}{\pi^2}  \int_0^{\pi/4}\int_0^\delta  \frac 12 \frac{\partial }{\partial r} \ln (\lambda + r^2) (1 + h) d r d\theta\\ 
 & =  \frac{4}{\pi^2} \int_0^{\pi/4} \left (  \ln (\lambda + r^2) (1+ h) |_{r=0}^{r=\delta} - \int_0^\delta  \ln (\lambda+ r^2) h'(r,\theta) dr \right) d\theta\\
 & = \frac{1}{\pi} \ln \frac{1}{\lambda} +I_{A_{*}} - I_{A_{**}},
  \end{align*} 
  where 
  \begin{align*} 
  I_{A_{*}} &=  \frac 4{\pi^2}\int_0^{\pi/4} \ln (\lambda + \delta^2) ( 1+ h(\delta,\theta) ) d \theta ; \mbox{ and } \\
  I_{A_{**}}&=\frac{4}{\pi^2} \int_0^{\pi/4}\int_0^\delta  \ln (\lambda+ r^2) h'(r,\theta) dr d\theta\mbox{ and } h'= \frac{\partial h}{\partial r}.
  \end{align*} 
  Clearly, $I_{A_*}$ is uniformly bounded.  As for $I_{A_{**}}$, first observe that  $h'$ is bounded  on $A$. As a result, 
 Now 
 $$ |I_{A_{**}}|\le c \int_0^\delta |\ln (\lambda + r^2)| dr.$$ 
 Note that if $\lambda = a + ib$, then $\ln (\lambda + r^2) =\frac 12  \ln |(a+r^2)^2 +b^2 |+ \alpha(\lambda,r)$ where $\alpha$ is equal to to $i$ times the argument of $\lambda + r^2$.  Thus, $|\ln (\lambda + r^2)| \le 2 |\ln |a+r^2|| + c$. If $a\ge 0$, $|\ln |a+r^2| |\le |\ln r^2|$, and the integrability of $\ln r^2$ over $[0,\delta]$ guarantees that  $I_{A_{**}}$ is uniformly bounded over $\lambda$ with nonnegative real part. If $a<0$,  then $a = -|a|$ and so $r^2 +a = (r- \sqrt{|a|})(r+\sqrt{|a|})$, so that $\ln |a+ r^2| \le  |\ln |r+\sqrt{|a|}| + |\ln |r-\sqrt{|a|}|\le \ln |r| + |\ln |r-|\sqrt{|a|}| $, and so,
 $$ \int_0^\delta | \ln |a+r^2|| dr = c + \int_0^{\delta} | \ln |r-\sqrt{|a|}|dr < c,$$ 
so that $I_{A_{**}}$ is again uniformly bounded over $\lambda$ with negative real part. \\
It remains to consider the imaginary part of $I(\lambda)$. It is easy to see that the imaginary part of $I_B =O(\Im \lambda)$. Thus, it remains to consider $\Im I_A$. Consider  $\lambda = -s + i\epsilon$ with $s\ge 0$, and leave the easier details for $s<0$ to the reader (in fact, this regime is not used in our paper). From \eqref{eq:2dI_polar} we observe that 
$$ -\Im I_A= \int_{A}  \frac{\epsilon} { (r^2 - s)^2 + \epsilon^2} (1+ h(r,\theta)) d r d\theta.$$ 
Now change variables to $u= r^2 -s$, then $r= \sqrt{s+u}$ and so $d r / du = \frac 12 \frac {1}{\sqrt{s+u}}$. Note also that $h(r,\theta) = h (\sqrt{s+u},\theta)$ jointly continuous and uniformly bounded on the domain of integration.  We then have that 
$$ -\Im I_A= \frac{4}{\pi^2}  \int_0^{\pi/4} \int_0^{\delta}  \frac{\epsilon}{\epsilon^2+ u^2}  (1+ h(\sqrt{s+u},\theta)) d u  d \theta .$$ 
Since $\frac{1}{\pi} \frac{\epsilon}{\epsilon^2+ u^2}$ is an approximation of the identity and $h$ is bounded on $A$, and remembering that $\im I_B = O(\epsilon)$,  it then follows from the dominated convergence theorem that 
\begin{align*} -\lim_{\epsilon\to 0} \Im I(\lambda) & = \frac 4 \pi \int_0^{\pi/4} (1+h (\sqrt{s},\theta)) d\theta\\
 & = \frac 4\pi \int_0^{\pi/4} \frac {d\theta }{{\sqrt{1-s \cos^2 \theta}}{\sqrt{1-s \sin^2 \theta}}}= 1+ o(1).
 \end{align*} 
  \end{enumerate}
          \end{proof}
             \begin{proof}[Proof of Lemma \ref{lem:martin}]
         We first show that $H_{-\infty} A_0=0$. To prove this observe that 
          $$ \Delta A_0(w) = \frac{1}{\pi^d} \int_{[0,\pi]^d}e^{i \scp{\varphi,w}}\sum_{|e|=1} \frac{1-e^{i\scp{\varphi,e}}}{ 2d  \Phi (\varphi)} d \varphi = \frac{1}{\pi^d} \int_{[0,\pi]^d}e^{i \scp{\varphi,w}} d \varphi =\delta_0(w).$$
          
We turn to uniqueness, which we prove according to dimension. Suppose $d \ge 3$.  Let $u$ be a positive  harmonic function for $H_{-\infty}$. Of course,    $u(0)=0$. It follows that $\Delta u (x) =  c_1 \delta_0(x) $, where $c_1= \frac{1}{2d} \sum_{|e|=1} u(e)>0$. Let now $k (x) = c_1 R_0 \delta_0 (x) $.  Then $\Delta (u+k)=0$. As is well-known,  the Martin boundary for $\Delta$ is spanned by $\ch$. Therefore $u+k = c_2 \ch$, or equivalently, $u= c_2 \ch -k$. Since $k$ is bounded, it follows that $u$ is bounded. Observing that $k (x) = c_3 P_x (\tau_0<\infty)$, and using the fact that $u(0)=0$, it follows that $u(x)= P_x (\tau_0=\infty)$, up to a positive multiplicative constant. 
    
    Next, suppose that $d=1,2$. We can rewrite \eqref{eq:killed_res} as 
    \begin{align}
    \nonumber
     R_\lambda^{-\infty}(x,y)&=-A_\lambda(x-y)+I(\lambda) - \frac{(I(\lambda)-A_\lambda(x))(I(\lambda)-A_\lambda(-y))}{I(\lambda)}\\
     \nonumber 
      &=-A_\lambda(x-y) + \frac{(A_\lambda(x)+A_\lambda(-y))I(\lambda)-A_\lambda(x)A_\lambda(-y)}{I(\lambda)}\\
\nonumber
       & = A_\lambda(x)+A_\lambda(-y)-A_\lambda(x-y) - \frac{A_\lambda(x)A_\lambda(-y)}{I(\lambda)}. 
      \end{align}   
    By recurrence,  $\lim_{\lambda \searrow 0} I(\lambda)=\infty$. Therefore  
       $$ R_{0}^{-\infty} (x,y) = A_0(x) + A_0(-y) - A_0(x-y).$$
     In addition, from the second equality in \eqref{eq:Alambda} we obtain that $\{ A_0(-y)-A_0(x-y):y \in \Z^d \}$ are the  Fourier coefficients  of a function in $L^1([0,2\pi]^d)$. The  Riemann-Lebesgue lemma implies then that  $\lim_{|y|\to \infty } A_0(-y) - A_0(x-y)=0$. Consequently, for any $x,x_0\in \Z^d-\{0\}$ we have 
     $$ \lim_{|y|\to \infty} \frac{R_0^{-\infty}(x,y)}{R_0^{-\infty}(x_0,y)}=\frac{A_0(x)}{A_0(x_0)},$$ 
     which proves that the Martin boundary for $H_{-\infty}$ is spanned by $A_0$. 
     \end{proof}

          \begin{proof}[Proof of Corollary \ref{cor:lasttimes}] ~
            \begin{enumerate}
       \item This follows immediately from Theorem \ref{th:Zbeta} as   
       $$E_{\beta,t} e^{-\lambda\I(t)}=Z_{\beta-\lambda,t}/Z_{\beta,t} \underset{t\to\infty}{\to} \frac{-\beta}{-\beta+\lambda},$$ which is the Laplace transform of an exponential random variable with rate $-\beta$. 
       \item Let $U$ be open subset of $(0,\infty)$. We will show that 
       $$ \liminf_{t\to\infty} P_{\beta,t} (\sigma_t \in U) \ge -\int_U \beta p_{\beta}(s,0,0) ds.$$
        Since by \eqref{eq:Zbeta} $-\beta p_{\beta} (s,0,0) ds$ is a probability density, this proves the claim. \\
        We first prove an auxiliary result. Suppose that $|e|=1$. Then clearly, 
          $$Z_{\beta,u}(e) = P_e (\tau >u) + \int_0^u Z_{\beta,u-s} d P_e(\tau \le s) \ge P_e(\tau>u)+Z_{\beta,u}P_e(\tau\le u).$$ 
      Dividing both sides by $Z_{\beta,u}$ and taking $u\to\infty$, we have 
 $$\psi_\beta(e) \ge \limsup_{u\to\infty} \frac{P_e(\tau >u)}{Z_{\beta,u}}+1.$$
  However, $\psi_\beta (e) =1-  \beta $, so that 
      $$\limsup_{u\to\infty} \frac{P_e(\tau >u)}{Z_{\beta,u}} \le - \beta.$$         
     $$ \frac{ Z_{\beta,u}(e)}{Z_{\beta,u}} = \frac{P_e (\tau>u)}{Z_{\beta,u}} + \int_0^ u \frac{Z_{\beta,u-s} }{Z_{\beta,u}} d P_e (\tau\le s).$$ 
     The left-hand side converges to $\psi_\beta(e)$, and therefore it follows from Fatou's lemma, that
      $$ \psi_\beta (e) \ge \liminf_{u\to\infty}  \frac{P_e (\tau>u)}{Z_{\beta,u}}  + \int_0^\infty d P_e (\tau\le s).$$ 
      From this, it follows that $\liminf_{u\to\infty}  \frac{P_e (\tau>u)}{Z_{\beta,u}}\ge - \beta$. 
      We conclude that
      \begin{equation} 
      \label{eq:how_far}
       \lim_{u\to\infty}  \frac{P_e (\tau>u)}{Z_{\beta,u}}=-\beta
       \end{equation}

        Suppose that $I= (a,b) \subset (0,\infty)$. Then we have 
        \begin{align} 
        \nonumber
          Z_{\beta,t} P_{\beta,t} (\sigma_t \in I)  &\ge E_0 \ch_0(X_a)e^{\beta \I (a)} \int_0^{b-a} e^{-\rho} e^{\beta \rho} P_{e} (\tau > t - a- \rho) d \rho\\
         \nonumber
          & = p_{\beta}(a,0,0) \int_0^{b-a} e^{-\rho+ \beta \rho} d\rho P_{e}(\tau > t-a-\rho) d\rho \\ 
          \label{eq:last_time} 
          & \ge p_{\beta}(a,0,0) (b-a) e^{-(b-a) (1+|\beta|)}P_e (\tau > t-b)
          \end{align} 
          For each $n$, partition $U$ into disjoint open  intervals each  of length $\le 2^{-n}$, with the $n+1$-th partition  embedded in the $n$-th partition (we omit a countable set on each partition). Let $f_n$ be the function which is constant on each element of the $n$-th partition. If this partition is given by the intervals $(a_{n,j},b_{n,j}),~j=1,\dots$, then 
          let $$f_n(s) = \sum_{j}\left (   p_{\beta}(a_{n,j} ,0,0)  e^{-(b_{n,j}-a_{n,j}) (1+|\beta|)}P_e (\tau > t-b_{n,j}\right) \ch_{(a_{n,j},b_{n,j})}(s).$$
          It follows from \eqref{eq:last_time} that 
          $$ P_{\beta,t} (\sigma_t \in U) \ge  \frac{1}{Z_{\beta,t}} \int_U f_n (s) ds.$$
          However, the continuity of $\beta_{\beta}(\cdot,0,0)$ and the fact that $P_e(\tau > t-\cdot)$ is nonincreasing, it follows that  $f_n (s) \to p_{\beta} (s,0,0) P_e (\tau>t-s)$ a.e. with respect to the Lebesgue measure. It then follows from Fatou's lemma that 
          $$P_{\beta,t} (\sigma_t \in U) \ge  \int_U p_{\beta}(s,0,0) \frac{ P_e (\tau> t-s)}{Z_{\beta,t}} ds.$$ 
          Applying Fatou's lemma again and \eqref{eq:how_far},  we obtain 
          $$ \liminf_{t\to\infty} P_{\beta,t} (\sigma_t \in U) \ge -\beta \int_U p_{\beta}(s,0,0) ds.$$ 
          The result follows. 
\item By Theorem \ref{th:small times} for every $T>0$, the distribution of the polymer measure  $P_{\beta,t}|_{{\cal F}_T}$ converges to the distribution of a Markov chain with generator $H_{\beta}^{\psi_{\beta}}$. By Theorem \ref{th:TR_theorem}, the latter is transient. In particular, the number of visits to the origin starting from the origin, $N_\infty$ is $\mbox{Geom}(\rho)$, with $\rho =Q_e(\tau_0<\infty)$, where $|e|=1$.  The weak convergence of the polymer measure and the continuity of $N_T$ guarantees that the distribution of $N_T$ also converges, and that the limit is stochastically dominated by $\mbox{Geom}(\rho)$. In addition, the transience of the limit, guarantees also that 
$\lim_{T\to\infty} \lim_{t\to\infty} P_{\beta,t} (N_T \in U) = Q_0 (N_\infty \in U)$. 
Next, observe that 
$$ P_{\beta,t} (N_t \in U) \ge  P_{\beta,t} (N_t \in U, \sigma_t \le T ) =P_{\beta,t} (N_T \in U) - P_{\beta,t} (\sigma_t > T).$$ 
However, by part ii, $\lim_{T\to\infty} \lim_{t\to\infty} P_{\beta,t} (\sigma_t>T)=0$. It then follows that 
$$\liminf_{t\to\infty}  P_{\beta,t} (N_t \in U) \ge Q_0( N_\infty \in U),$$ 
which proves that the distribution of $N_t$ under $P_{\beta,t}$ converges to $\mbox{Geom}(\rho)$. It remains to find $\rho$. 
We have 
$$ \lim_{t\to\infty}E^Q_0 \I(t) = E^Q_0 \sum_{j=0}^{N_\infty} J_j,$$
where $J_j$ are IID exponential random variables with rate $1-\beta$, the rate of jump from the origin by $Q$. However, the right-hand side is also equal to 
$$\frac{1}{\psi_\beta (0)} \int_0^\infty  p_{\beta}(s,0,0) \psi_{\beta}(0) ds   = I^\beta (0).$$ 
 But by \eqref{eq:Ilambda} and the Lemma \ref{lem:Ilambda} which guarantees that $\lim_{\lambda\to 0} I(\lambda)=\infty$, it follows that $ I^\beta (0)= \frac{1}{-\beta}$. Thus, 
$$ \frac{1}{-\beta} = \frac{E^Q_0 N_\infty}{1-\beta}= \frac{1}{\rho(1-\beta)},$$ and the result follows. 
           \end{enumerate}
     \end{proof}
       \subsection*{Acknowledgement} 
              The first author would like to express his deep gratitude to Mike Cranston and Stanislav Molchanov for introducing him to the model and their work at an early stage. 
  \providecommand{\bysame}{\leavevmode\hbox to3em{\hrulefill}\thinspace}
\providecommand{\MR}{\relax\ifhmode\unskip\space\fi MR }
\providecommand{\MRhref}[2]{%
  \href{http://www.ams.org/mathscinet-getitem?mr=#1}{#2}
}
\providecommand{\href}[2]{#2}

\end{document}